\documentclass{article}
\usepackage{amsmath,amssymb}
\DeclareMathOperator{\ME}{\mathsf{E}}
\DeclareMathOperator{\Prob}{\mathsf{P}}
\DeclareMathOperator{\cov}{cov}
\DeclareMathOperator{\var}{var}
\DeclareMathOperator{\sign}{sign}
\DeclareMathOperator{\Res}{Res}

\DeclareMathOperator*{\closedspan}{\overline{\textrm{span}}}
\DeclareMathOperator{\colspan}{span}
\usepackage{amsthm}
\newtheorem{proposition}{Proposition}
\newtheorem{corollary}{Corollary}
\newtheorem{lemma}{Lemma}

\newtheorem*{proposition-den-suf-alt}{Proposition \ref{prop:existance-eX-newps}\textsuperscript{\boldmath$\prime$}}
\theoremstyle{remark}
\newtheorem{remark}{Remark}
\theoremstyle{definition}

\newtheorem{definition}{Definition}
\begin{document}
\title{First-order planar autoregressive model}
\author{Sergiy Shklyar}
\maketitle
\begin{abstract}
  This paper establishes the conditions of existence of a stationary solution
  to the first order autoregressive equation on a plane
  as well as properties of the stationarity solution.
The first-order autoregressive model on a plane is
defined by the equation
\[
  X_{i,j} = a X_{i-1,j} + b X_{i,j-1} + c X_{i-1,j-1} + \epsilon_{i,j}.
\]
A stationary solution $X$ to the equation exists if and only if
$(1-a-b-c)\*(1-a+b+c)\*(1+a-b+c)\*(1+a+b-c)>0$.
The stationary solution $X$ satisfies the causality condition
with respect to the white noise $\epsilon$ if and only if
$1-a-b-c>0$, $1-a+b+c>0$, $1+a-b+c>0$ and $1+a+b-c>0$.
%We also found a sufficient condition for $X$ to be purely nondeterministic.
A sufficient condition for $X$ to be purely nondeterministic is provided.

% We provide an explicit expression for the autocovariance function of $X$
% at some points, and prove Yule--Walker equations\,---\,\allowbreak
% recursive equations that allow us to compute the autocovariance function
% everywhere. We describe all situations where different parameters determine
% the same autocovariance function of $X$.

An explicit expression for the autocovariance function of $X$
at some points is provided. With Yule--Walker equations,
this allows to compute the autocovariance function everywhere.
In addition,
all situations are described where different parameters determine the same autocovariance function of $X$.

\end{abstract}

\section{Introduction}
\paragraph{The model.}
Let $\epsilon_{i,j}$
be a uncorrelated zero-mean equal-variance variables,
$\ME \epsilon_{i,j}=0$ and $\ME \epsilon_{i,j}^2 = \sigma_\epsilon^2 > 0$.
Consider the equation
\begin{equation}\label{eq:difeq1}
  X_{i,j} = a X_{i-1,j} + b X_{i,j-1} + c X_{i-1,j-1} + \epsilon_{i,j}
\end{equation}
with $a$, $b$ and $c$ fixed coefficients, and $X$ an unknown array
defined on an integer lattice\,---\,\allowbreak for all integer indices.
% We are considering the following topics.
% For what $a$, $b$ and $c$ is equation~\eqref{eq:difeq1} stable?
% When does \eqref{eq:difeq1} have a solution $X$ that is a stationary random field?
% We also demonstrate some properties of the stationary solution.

Eq.~\eqref{eq:difeq1} can be used for modelling and simulating random images.

\paragraph{Subject of the study.}
We consider the following questions about the equation~\ref{eq:difeq1}.
\begin{itemize}
  \item
    For what coefficients $a$, $b$ and $c$ does a stationary solution of the \eqref{eq:difeq1} exist?
  \item
    How to compute the autocovariance function of the stationary solution?
  \item
    For what coefficients $a$, $b$ and $c$ does the stationary solution satisfy the causality condition
    with respect to $\epsilon$, that is, can be represented in the form
    \[
      X_{i,j} = \sum_{k=0}^\infty \sum_{l=0}^\infty \psi_{k,l} \epsilon_{i-k, j-l}
    \]
    with summable coefficients $\psi_{k,l}$?
  \item
    How the causality condition is related to the stability of equation~\eqref{eq:difeq1} taken nondeterministically?
  \item
    Is the stationary solution purely nondeterministic?
    This property is related to the causality condition with respect to some uncorrelated random field,
    called innovations:
    does the field of innovations exist?
  \item Which different parameters $(a, b, c, \sigma_\epsilon^2)$ determine the same autocovariance function of the field $X$?
\end{itemize}

\paragraph{Previous art.}
The review of statistical models for representing a discretely observed field on a plane,
including regression models and trend-stationary models, in given in \cite{UnwinHepple:1974}.
The study of planar models, their estimation and applications is presented in the monograph \cite{Cressie:1993}.

A theory of stationary random field on $n$-dimensional integer lattice
was constructed by Tj\o stheim \cite{Tjostheim:1978}.
For the stationary fields, he considers the property of being
``purely nondeterministic,'' which is related to the existence
of appropriate field of innovations.
He considers ARMA models, for which he provides conditions
for stability of the equation and existence of a stationary solution.
For AR models, he establishes Yule--Walker equations,
and also consistency and asymptotic normality of the Yule--Walker estimator.

Autoregressive models in a $n$-dimensional space had been studied before Tj\o stheim.
Whittle \cite{Whittle:1954} had been studied such topics of autoregressive model in the plane
as the spectral density of the stationary solution, nonidentifiability,
and estimation of the regression coefficients.
In theoretic part, the order of the model was not limited;
in the example, he models data with the second-order model of the form
\[
  \xi_{s,t} = \alpha \xi_{s+1,t} + \beta \xi_{s-1,t}
  + \gamma \xi_{s,t+1} + \delta \xi_{s,t-1} + \epsilon_{s,t} .
\]
The model equivalent to \eqref{eq:difeq1} is considered in \cite{Pickard:1980,ToryPickard:1992}.
One of simpler planar models considered in the paper is defined by equation
$
  \xi_{s,t} = \alpha \xi_{s+1,t} + \beta \xi_{s,t+1} + \epsilon_{s,t},
$
which, up to different notation, is also a spectral case of \eqref{eq:difeq1}.

The autocovariance function of a stationary field is even-symmetric, that is
$\gamma_X(h_1,h_2) = \gamma_X(-h_1,-h_2)$;
however, $\gamma_X(h_1,h_2)$ need not be an even function in each of its arguments.
The autocovariance function of the stationary stationary solution
to \eqref{eq:difeq1} is even each argument if the coefficients satisfy
$c = - a b$.
This special case, its estimation and applications are considered
in \cite{Baran:2004,Martin:1979,Jain:1981,Martin:1990,Martin:1996}.

The property of random field of being \textit{purely nondeterministic} is introduced in \cite{Tjostheim:1978}.
Further discussion of this property is found in \cite{Kallianpur:1981} and \cite{Tjostheim:1983}.

\paragraph{Methods.}
We have obtained most results by manipulation with the spectral density.

\paragraph{Structure of the paper.}
In Section~2, we provide necessary and sufficient conditions
for existence of the stationary solution to equation~\eqref{eq:difeq1}.
We also prove that the stationary solution must be centered and
must have a spectral density.
In Section~3, we study the properties of the autocovariance function
$\gamma_X(h_1,h_2)$ of the stationary solution.
We find $\gamma_X(h_1,h_2)$ for $h_1 = 0$ and for $h_2 = 0$,
and we prove Yule--Walker equations.
This allows us to evaluate  $\gamma_X(h_1,h_2)$ everywhere.
Section~\ref{sec:stability-causality} is dedicated to the case where the equation~\eqref{eq:difeq1} is stable.
In Section~\ref{sec:symmetry} we show nonidentifiability, there the same autocovariance function
(and, in Gaussian case, the same distribution) of the random field $X$
corresponds to two different values of the parameter vectors.
In Section~\ref{sec:pnd} we study Tj{\o}stheim's pure nondeterminism.
Section~\ref{sec:conclusion} concludes. Some auxiliary statements are presented in the appendix,
with is in Section~\ref{sec:appendix}.

\section{Existence of a stationary solution}
\subsection{Spectral density}
\label{sec:sdens}
Denote the shift operators
\begin{equation}\label{eq:defLxLy}
  L_x X_{i,j} = X_{i-1,j}
  \quad \mbox{and} \quad
  L_y X_{i,j} = X_{i,j-1} .
\end{equation}
Eq.~\eqref{eq:difeq1} rewrites as follows
\begin{equation}\label{eq:difeq1a}
  X - a L_x X - b L_y X - c L_x L_y X = \epsilon .
\end{equation}
We use the following definition of a spectral density. An integrable function
$f_X(\nu_1, \nu_2)$ is a spectral density for the wide-sense stationary field
$X$ if
\begin{align}
  \gamma_X(h_1,\,h_2)
  &= \cov(X_{i,j},X_{i-h_1,j-h_2})
  = \nonumber \\ &=
  \int_{-1/2}^{1/2} \int_{-1/2}^{1/2} \exp(2 \pi {\mathrm{i}} \, (h_1 \nu_1 + h_2 \nu_2) )
  f_X(\nu_1,\nu_2) \, d\nu_1 \, d\nu_2 ,
  \label{eq:def-sd-1}
\end{align}
where the upright $\mathrm{i}$ is the imaginary unit.
This can be rewritten without complex numbers as
\begin{equation}
  \label{eq:def-sd-2}
  \gamma_X(h_1,h_2)
  %= \cov(X_{i,j},X_{i-h_1,j-h_2})
  = \int_{-1/2}^{1/2} \int_{-1/2}^{1/2} \cos(2 \pi  \, (h_1 \nu_1 + h_2 \nu_2) )
  f_X(\nu_1,\nu_2) \, d\nu_1 \, d\nu_2
\end{equation}
with further restraint that the spectral density should be an even function,
$f_X(\nu_1, \nu_2) = f_X(-\nu_1, -\nu_2)$.

Due to \eqref{eq:difeq1a}, the spectral densities of the random fields
$\epsilon$ and $X$, if they exist, are related as follows
\begin{equation}\label{eq:fepsfX}
  f_\epsilon(\nu_1, \nu_2) =
  |1 - a e^{2\pi \mathrm{i} \nu_1} - b e^{2 \pi \mathrm{i} \nu_2}
  - c e^{2 \pi \mathrm{i} (\nu_1 + \nu_2)}|^2 \,
  f_X(\nu_1, \nu_2) .
\end{equation}
The random field $\epsilon$ is a white noise with
constant spectral density $f_\epsilon(\nu_1, \nu_2) = \sigma_\epsilon^2$.
Thus, the spectral density of the stationary field $X$, if it exists, is equal to
\begin{equation}\label{eq:sdensX}
  f_X(\nu_1, \nu_2) =
  \frac{\sigma_\epsilon^2}
  {|1 - a e^{2\pi \mathrm{i} \nu_1} - b e^{2 \pi \mathrm{i} \nu_2}
  - c e^{2 \pi \mathrm{i} (\nu_1 + \nu_2)}|^2} \, .
\end{equation}

\begin{lemma}\label{lemma:denominator}
  Denote
  \[
    D = (1 - a - b - c) (1 - a + b + c) (1 + a - b + c) (1 + a + b - c) .
  \]
  Consider the denominator in \eqref{eq:sdensX}
  \[
    g(\nu_1,\nu_2) = |1 - a e^{2\pi \mathrm{i} \nu_1} - b e^{2 \pi \mathrm{i} \nu_2}
     - c e^{2 \pi \mathrm{i} (\nu_1 + \nu_2)}|^2,
     \qquad \nu_1, \nu_2 \in \left( -{\textstyle\frac12}, \, {\textstyle\frac12} \right] .
  \]
  It attains zero if and only if $D\le 0$.
  It can attain zero not more than at $2$ points on the half-open box
  $\left( -\frac12, \frac12 \right]^2$.
\end{lemma}
\begin{proof}
  The following lines are equivalent:
  \begingroup
  \allowdisplaybreaks[2]
  \begin{gather}
    \exists \nu_1 \in \left(-{\textstyle\frac12}, {\textstyle\frac12}\right] \;
    \exists \nu_2 \in \left(-{\textstyle\frac12}, {\textstyle\frac12}\right] \: : \:
     |1 - a e^{2\pi \mathrm{i} \nu_1} - b e^{2 \pi \mathrm{i} \nu_2}
     - c e^{2 \pi \mathrm{i} (\nu_1 + \nu_2)}|^2 = 0,
    \nonumber
    \\
    \exists \nu_1 \in \left(-{\textstyle\frac12}, {\textstyle\frac12}\right] \;
    \exists \nu_2 \in \left(-{\textstyle\frac12}, {\textstyle\frac12}\right] \: : \:
     1 - a e^{2\pi \mathrm{i} \nu_1} = e^{2 \pi \mathrm{i} \nu_2}
     (b- c e^{2 \pi \mathrm{i} \nu_1}),
    \label{eq:profe0-2}
    \\
    \exists \nu_1 \in \left(-{\textstyle\frac12}, {\textstyle\frac12}\right] \: : \:
     |1 - a e^{2\pi \mathrm{i} \nu_1}| = |(b - c e^{2 \pi \mathrm{i} \nu_1})|,
    \nonumber
    \\
    \exists \nu_1 \in \left(-{\textstyle\frac12}, {\textstyle\frac12}\right] \: : \:
     |1 - a e^{2\pi \mathrm{i} \nu_1}|^2 = |(b - c e^{2 \pi \mathrm{i} \nu_1})|^2,
    \nonumber
    \\
    \exists \nu_1 \in \left(-{\textstyle\frac12}, {\textstyle\frac12}\right] \: : \:
     1 - 2 a \cos(2\pi  \nu_1) + a^2  = b^2 - 2 b c \cos(2 \pi  \nu_1) + c^2,
    \nonumber
    \\
    \exists \nu_1 \in \left(-{\textstyle\frac12}, {\textstyle\frac12}\right] \: : \:
     1 + a^2 - b^2 - c^2  = 2 (a - b c) \cos(2 \pi  \nu_1),
    \label{eq:profe0-6}
    \\
     |1 + a^2 - b^2 - c^2|  \le 2 \, |a - b c|,
    \nonumber
    \\
     (1 + a^2 - b^2 - c^2)^2 - 4 (a - b c)^2 \le 0,
    \nonumber
    \\
     D \le 0 .
    \nonumber
  \end{gather}\endgroup
  The sufficient and necessary condition for the denominator $g(\nu_1,\nu_2)$
  to attain $0$ is proved.

  Now treat $a$, $b$ and $c$ as fixed, and assume that $D\le 0$.
  Equality in \eqref{eq:profe0-6} can be attained for not more than two
  $\nu_1 \in \left(-\frac12, \frac12\right]$.
  For each $\nu_1$, equality in \eqref{eq:profe0-2} can be attained for no more than
  one
  $\nu_2 \in \left(-\frac12, \frac12\right]$.
  Thus, the function
  $
      1 - a e^{2\pi \mathrm{i} \nu_1} - b e^{2 \pi \mathrm{i} \nu_2}
     - c e^{2 \pi \mathrm{i} (\nu_1 + \nu_2)}
  $
  has not more
  than 2 zeros in the half-open box
  $\left(-\frac12, \frac12\right]^2$.
  The denominator $g(\nu_1,\nu_2)$ cannot attain zero at more than two points either.
\end{proof}

\subsection{Necessary and sufficient conditions for the existence
of a stationary solution}

\begin{proposition}\label{prop:existance-eX-newps}
  Let $\sigma^2>0$, $a$, $b$ and $c$ be fixed real numbers.
  A collection of $\epsilon = \{\epsilon_{i,j},\; i,j\in\mathbb{Z}\}$
  of uncorrelated random variables with zero mean and equal variance $\sigma^2$
  and a wide-sense stationary random field
  $X = \{X_{i,j},\; i,j\in\mathbb{Z}\}$
  that satisfy equation \eqref{eq:difeq1}
  exist if and only if $D>0$, where $D$ comes from Lemma~\ref{lemma:denominator}.
\end{proposition}
\begin{proof}
  Denote
  \[
    g_1(\nu_1, \nu_2) = 1 - a e^{2\pi \mathrm{i} \nu_1} - b e^{2 \pi \mathrm{i} \nu_2}
     - c e^{2 \pi \mathrm{i} (\nu_1 + \nu_2)};
  \]
  then $g(\nu_1,\nu_2)$ defined in Lemma~\ref{lemma:denominator} equals $g(\nu_1,\nu_2) = |g_1(\nu_1,\nu_2)|^2$.
  \textit{Necessity.}
  Let random fields $\epsilon$ and $X$ that satisfy equation \eqref{eq:difeq1}
  and the other conditions of Proposition~\ref{prop:existance-eX-newps}.
  Both $\epsilon$ and $X$ are wide-sense stationary. Hence, they have spectral measures --- denote them
  $\lambda_\epsilon$ and $\lambda_X$ --- and $\lambda_\epsilon$ is a constant-weight Lebesgue measure
  in $\left(-\frac12, \frac12\right]^2$.
  Because of \eqref{eq:difeq1a}, $\lambda_\epsilon$ has a Radon--Nikodym density $g(\nu_1,\nu_2)$ with respect to (w.r.t.) $\lambda_X$:
  \[
    \frac{d \lambda_\epsilon}{d \lambda_X} (\nu_1, \nu_2) = |1 -
     |1 - a e^{2\pi \mathrm{i} \nu_1} - b e^{2 \pi \mathrm{i} \nu_2}
     - c e^{2 \pi \mathrm{i} (\nu_1 + \nu_2)}| .
  \]

  Now assume that $D\le 0$. Then, according to Lemma~\ref{lemma:denominator},
  $g(\nu_1, \nu_2)$ attains 0 at one or two points on a half-open box $\left(-\frac12,\frac12\right]^2$.
  Denote
  \begin{gather*}
    A_1 = \left\{ (\nu_1,\nu_2) \mathbin{\in} \left( -{\textstyle\frac12}, \, {\textstyle\frac12}\right]^2 : g_1(\nu_1,\nu_2) \neq 0 \right\}
        = \left\{ (\nu_1,\nu_2) \mathbin{\in} \left( -{\textstyle\frac12}, \, {\textstyle\frac12}\right]^2 : g(\nu_1,\nu_2) > 0 \right\}, \\
    A_2 = \left\{ (\nu_1,\nu_2) \mathbin{\in} \left( -{\textstyle\frac12}, \, {\textstyle\frac12}\right]^2 : g_1(\nu_1,\nu_2) = 0 \right\}
        = \left\{ (\nu_1,\nu_2) \mathbin{\in} \left( -{\textstyle\frac12}, \, {\textstyle\frac12}\right]^2 : g(\nu_1,\nu_2) = 0 \right\}; \\
    \lambda_1(A) = \lambda_X(A \cap A_1),
    \qquad
    \lambda_2(A) = \lambda_X(A \cap A_2)
  \end{gather*}
  for all measurable sets $A\subset\left( -{\textstyle\frac12}, \, {\textstyle\frac12}\right]^2$.

  The set $A_2$ contains one or two points. As $\lambda_\epsilon$ is absolutely continuous,
  $\lambda_\epsilon(A) = \lambda_\epsilon(A \setminus A_2) = \lambda_\epsilon(A \cap A_1)$.
  Measures $\lambda_1$ and $\lambda_\epsilon$ are absolutely continuous with respect of each other:
  $d \lambda_1 / d \lambda_\epsilon\, (\nu_1, \nu_2) = g(\mu_1,\nu_2)^{-1} $.
  Thus, $\lambda_1$ is absolutely continuous with respect to Lebesgue measure:
  \[
    \lambda_1(A) = \iint_{A \cup (-1/2,1/2]^2} \frac{\sigma_\epsilon^2} {g(\nu_1,\nu_2)} \, d\nu_1 \, d\nu_2.
  \]
  The measure $\lambda_2$ is concentrated at not more than 2 points.
  Thus, $\lambda_1$ and $\lambda_2$ are absolutely continuous and discrete components of the measure $\lambda_X$,
  and non-discrete singular component is zero.

  Let $(\nu_1^{(0)}, \nu_2^{(0)})$ be one of the points of $A_2$. As $g_1(\nu_1,\nu_2)$ is differentiable and $g(\nu_1,\nu_2)\ge 0$,
  at the neighbourhood of $(\nu_1^{(0)}, \nu_2^{(0)})$
  \begin{gather*}
    g_1(\nu_1, \nu_2) = o(|\nu_1 - \nu_1^{(0)}| + |\nu_2 - \nu_2^{(0)}|), \\
    g(\nu_1, \nu_2) = |g_1(\nu_1, \nu_2)|^2 = o((\nu_1 - \nu_1^{(0)})^2 + (\nu_2 - \nu_2^{(0)})^2), \\
    \frac{1}{g(\nu_1, \nu_2)} > \frac{\textrm{const}}{(\nu_1 - \nu_1^{(0)})^2 + (\nu_2 - \nu_2^{(0)})^2}
  \end{gather*}
  for some $\textrm{const} > 0$;
  \[
    \int_{-1/2}^{1/2} \int_{-1/2}^{1/2} \frac{1}{g(\nu_1,\nu_2)} \, d\nu_1 \, d\nu_2 = +\infty .
  \]
  Hence
  \[
    \lambda_X\left( \left(-{\textstyle\frac12}, {\textstyle\frac12}\right]^2 \right)
    \ge
    \lambda_1\left( \left(-{\textstyle\frac12}, {\textstyle\frac12}\right]^2 \right)
    =
    \iint_{(-1/2,1/2]^2} \frac{\sigma^2}{g(\nu_1,\nu_2)} \, d\nu_1 \, d\nu_2 = \infty.
  \]
  Thus, the spectral measure of the random field $X$ is infinite, which is impossible.
  The assumption $D\le 0$ brings the contradiction. Thus, $D>0$.

  \paragraph{\textit{\textmd{Sufficiency.}}}
  Let $D<0$.  Then $\sigma_\epsilon^2 / g(\nu_1, \nu_2)$ is an even integrable function
  on $\left[-\frac12, \frac12\right]^2$ that attains only positive values.
  Then there exists a zero-mean Gaussian stationary random field $X$ with spectral density
  $\sigma_\epsilon^2 / g(\nu_1,\nu_2)$, see Lemma~\ref{lem:existence-field-dessdens}.
  Define $\epsilon$ by formula \eqref{eq:difeq1}. Then $\epsilon$ is zero-mean and stationary;
  $\epsilon$ has a constant spectral density $\sigma_\epsilon^2$. Thus, $\epsilon$ is
  a collection of uncorrelated random variables with zero mean and variance $\sigma_\epsilon^2$.
  Thus, the random fields $\epsilon$ and $X$ satisfy the desired conditions.
\end{proof}

\begin{remark}
  In the sufficiency part of Proposition~\ref{prop:existance-eX-newps}, the probability space is specially constructed.
  The random fields $X$ and $\epsilon$ constructed are jointly Gaussian and jointly stationary, which means that
  $\{(X_{i,j}, \epsilon_{i,j})\; i,j\mathbin{\in}\mathbb{Z}\}$
  is a two-dimensional Gaussian stationary random field on a two-dimensional lattice.
\end{remark}

\begin{corollary}\label{cor:hassectrdensity}
  Let $X$ be a wide-sense stationary field that satisfies \eqref{eq:difeq1}.
  Then $X$ is centered and has a spectral density \eqref{eq:sdensX}.
\end{corollary}
\begin{proof}
  Due to Proposition~\ref{prop:existance-eX-newps}, $D>0$ (where $D$ comes from Lemma~\ref{lemma:denominator});
  Hence, $1 - a - b - c \neq 0$. Taking expectation in \eqref{eq:difeq1} immediately imply that the mean of
  the stationary random field $X$ is zero.

  Demonstration that $X$ has a spectral density partially repeat the \textit{Necessity} part of the proof
  of Proposition~\ref{prop:existance-eX-newps};
  in what follows, we use notation $g(\nu_1,\nu_2)$, $\lambda_X$, $A_2$, $\lambda_1$ and $\lambda_2$ from that proof.
  According to Lemma~\ref{lemma:denominator}, $D>0$ imply that $A_2=\emptyset$.
  The discrete component $\lambda_2$ of the spectral measure $\lambda_X$ of the random filed $X$ is concentrated
  on the set $A_2$; thus, $\lambda_2 = 0$.
  The non-discrete singular component of $\lambda_X$ is also zero.
  The spectral measure $\lambda_X = \lambda_1 + \lambda_2$ is absolutely continuous, that is the random field
  $X$ has a spectral density.
\end{proof}

\subsection{Restatement of the existence theorem}
\begin{proposition}\label{prop:existance-eX-alt}
  Let $\sigma^2>0$, $a$, $b$ and $c$ be fixed real numbers,
  and $\epsilon_{i,j}$
  be a collection of uncorrelated random variables with zero mean and equal variance $\sigma^2$.
  A wide-sense stationary random field
  $X$
  that satisfy equation \eqref{eq:difeq1}
  exists if and only if $D>0$, where $D$ comes from Lemma~\ref{lemma:denominator}.
\end{proposition}
\begin{proof}
  \textit{Necessity} is logically equivalent to one in Proposition~\ref{prop:existance-eX-newps}.

  \textit{Sufficiency.}
  Denote
  \begin{equation}\label{eq:psi-as-int}
    \psi_{k,l}
    =
      \int_{-1/2}^{1/2} \int_{-1/2}^{1/2}
      \frac{\exp(-2 \pi \mathrm{i} (k \nu_1 + l \nu_2))}
      {1 - a e^{2\pi \mathrm{i} \nu_1}
         - b e^{2\pi \mathrm{i} \nu_2} - c e^{2\pi \mathrm{i} (\nu_1+\nu_2)}}
      \, d\nu_1 \, d\nu_2\,.
  \end{equation}
  The integrand is well-defined due to Lemma~\ref{lemma:denominator}.
  It is infinitely times differentiable periodic function.
  Due to Lemma~\ref{lemma:suf-four-summ},
  \begin{equation}\label{eq:psi-summable-662}
    \sum_{k=-\infty}^\infty \sum_{l=-\infty}^\infty |\psi_{k,l}| < \infty .
  \end{equation}

  Now prove that
  \begin{equation}\label{eq:def-X-4ser}
    X_{i,j} = \sum_{k=-\infty}^\infty  \sum_{l=-\infty}^\infty
    \psi_{k,l} \epsilon_{i-k,j-l}
  \end{equation}
  is a solution.
  The series in the is convergent in mean squares as well as almost surely
  due to Proposition~\ref{prop:absolute-double-series}.
  The resulting field $X$ is stationary as a linear transformation
  of a stationary field $\epsilon$.

  By changing the indices,
  \begin{align*}
    X_{i,j} &- a X_{i-1,j} - b X_{i,j-1} - c X_{i-1,j-1}
    = \\ &=
    \sum_{k=-\infty}^\infty \sum_{l=-\infty}^\infty \psi_{k,l}
    (\epsilon_{i-k,j-l} - a \epsilon_{i-1-k,j-l} - b \epsilon_{i-k,j-1-l} - c \epsilon_{i-1-k,j-1-l})
    = \\ &=
    \sum_{k=-\infty}^\infty \sum_{l=-\infty}^\infty
    (\psi_{k,l} - a \psi_{k-1,l} - b \psi_{k,l-1} - c \psi_{k-1,l-1}) \epsilon_{i-l,j-1} .
  \end{align*}
  This transformation of the Fourier coefficients corresponds to multiplication
  of the integrand by a certain function. Thus,
  \begin{gather*}
    \begin{aligned}
      \psi_{k,l} - a \psi_{k-1,l} &- b \psi_{k,l-1} - c \psi_{k-1,l-1}
      = \\ &=
      \int_{-1/2}^{1/2} \int_{-1/2}^{1/2}
      \exp(-2 \pi \mathrm{i} (k \nu_1 + l \nu_2))
      \, d\nu_1 \, d\nu_2
    =
    \begin{cases}
      1 \mbox{~if~$k=l=0$}, \\
      0 \mbox{~otherwise};
    \end{cases}
    \end{aligned} \\
    X_{i,j} - a X_{i-1,j} - b X_{i,j-1} - c X_{i-1,j-1} = \epsilon_{i,j},
  \end{gather*}
  and the random field $X$ is indeed a solution to \eqref{eq:difeq1}.
  Thus, $X$ is a desired random field.
\end{proof}

\section{The autocovariance function}
\label{sec:theacf}
Assume that $\{X_{i,j}, i, j \in\mathbb{Z}\}$
is a wide-sense stationary field that satisfies \eqref{eq:difeq1}.
According to Corollary~\ref{cor:hassectrdensity},
$X$ has a spectral density, which is defined by \eqref{eq:sdensX}.
The autocovariance function can be evaluated with
\eqref{eq:def-sd-1} or \eqref{eq:def-sd-2}.
The explicit formula is
\begin{align}
  \gamma_X(h_1,h_2)
  &=
  \int_{-1/2}^{1/2} \int_{-1/2}^{1/2}
  \frac{\sigma_\epsilon^2 \exp(2\pi\mathrm{i} (h_1 \nu_1 + h_2 \nu_2))}
  {|1 - a e^{2\pi\mathrm{i} \nu_1} - b e^{2\pi\mathrm{i} \nu_2} - c e^{2 \pi \mathrm{i} (\nu_1 + \nu_2)}|^2}
  \, d\nu_1 \, d\nu_2
  \nonumber
  = \\ &=
  \int_{-1/2}^{1/2} \int_{-1/2}^{1/2}
  \frac{\sigma_\epsilon^2 \exp(2\pi\mathrm{i} (h_1 \nu_1 + h_2 \nu_2))}
  {g(\nu_1, \nu_2)}
  \, d\nu_1 \, d\nu_2,
  \label{eq:acovf-via-integral-expl}
\end{align}
with $g(\nu_1, \nu_2)$ defined in Lemma~\ref{lemma:denominator}.

For fixed $h_1$ and $h_2$, it is possible to compute the integral in \eqref{eq:acovf-via-integral-expl}
as a function of $a$, $b$ and $c$.
\subsection{The autocovariance function for $h_1 = 0$.}
\label{ss:h1eq0}
The denominator $g(\nu_1,\nu_2)$ in \eqref{eq:sdensX} is equal to
\begin{align*}
    g(\nu_1, \nu_2) &= 1 + a^2 + b^2 + c^2 + 2 (ac - b) \cos (2 \pi \nu_2)
    + \mbox{} \\ & \quad +
    2 ((bc - a) + (ab - c) \cos(2 \pi \nu_2)) \cos(2 \pi \nu_1)
    + \mbox{} \\ & \quad +
    2 (ab + c) \sin(2 \pi  \nu_1) \sin(2 \pi \nu_2)
    = \\ &=
    A + B \cos(2 \pi \nu_1) + C \sin(2 \pi \nu_1),
\end{align*}
  with
  \begin{equation}\label{eq:ABC-denom}
    \begin{gathered}[c]
      A = 1 + a^2 + b^2 + c^2 + 2 (ac - b) \cos (2 \pi \nu_2), \\
      B = 2 ((bc - a) + (ab - c) \cos(2 \pi \nu_2)), \qquad
      C = 2 (ab + c) \sin(2 \pi \nu_2)
    \end{gathered}
  \end{equation}

  We are going to use Lemma~\ref{lemma:integration}
  to calculate the integral $\int_{-1/2}^{1/2} g(\nu_1,\nu_2)^{-1} \, d\nu_1$.
  We have to verify the conditions $A>0$ and $A^2 - B^2 - C^2 > 0$:
\begin{equation}\label{eq:Ap-1}
  (1 + a^2 + b^2 + c^2)^2 - 4 (ac - b)^2 =
  %  (1 + a^2 + b^2 + c^2 - 2ac + 2b) (1 + a^2 + b^2 + c^2 + 2ac - 2b) =
  ((a - c)^2 + (b+1)^2) ((a+c)^2 + (b-1)^2) \ge 0.
\end{equation}
The zero is attained if either $a=c$ and $b=-1$, or $a+c=0$ and $b=1$.
In both cases $D=0$, while $D>0$ according to  Proposition~\ref{prop:existance-eX-newps}
($D$ is defined in Lemma~\ref{lemma:denominator}).
Thus the inequality in \eqref{eq:Ap-1} is strict,
\begin{gather*}
  (1 + a^2 + b^2 + c^2)^2 > 4 (ac - b)^2, \\
  1 + a^2 + b^2 + c^2 > |2 (ac - b)| \ge -2 (ac - b) \cos(2 \pi \nu_2), \\
  A = 1 + a^2 + b^2 + c^2 + 2(ac-b) \cos(2\pi\nu_2) > 0 .
\end{gather*}
With some trigonometric transformations,
\[
  A^2 - B^2 - C^2 = (1 - a^2 + b^2 - c^2 - 2 (ac + b) \cos(2 \pi \nu_2) )^2.
\]
Again, with some transformations and notion that $D>0$,
\begin{gather*}
   (1 - a^2 + b^2 - c^2)^2 - 4 (ac + b)^2 = D > 0, \\
   |1 - a^2 + b^2 - c^2| > |2 (ac + b)| > 2 (a c + b) \cos(2 \pi \nu_2), \\
   1 - a^2 + b^2 - c^2 \neq 2 (a c + b) \cos(2 \pi \nu_2), \\
   A^2 - B^2 - C^2 = (1 - a^2 + b^2 - c^2 - 2 (ac + b) \cos(2 \pi \nu_2) )^2 > 0.
\end{gather*}
According to Lemma~\eqref{lemma:integration}, eq.~\eqref{eq:li-eq1},
\[
  \int_{-1/2}^{1/2} \frac{d \nu_1}{g(\nu_1,\nu_2)} = \frac{1}{\sqrt{A^2 - B^2 - C^2}}
  = \frac{1}{|1 - a^2 + b^2 - c^2 - 2 (ac + b) \cos(2 \pi \nu_2)|} .
\]

Denote
\[
  g_2(\nu_2) = 1 - a^2 + b^2 - c^2 - 2 (ac + b) \cos(2 \pi \nu_2)
\]
and compute $\int_{-1/2}^{1/2} e^{2\pi\mathrm{i}\nu_2 h_2} \,|g_2(\nu_2)|^{-1} \, d\nu_2$ using formula~\eqref{eq:li-eq1}.
The expression $g_2(\nu_2)$ does not attain $0$;
hence, either $g_2(\nu_2)>0$ for all $\nu_2$, or $g_2(\nu_2)<0$ for all $\nu_2$.
Then
\[
  |g_2(\nu_2)| = A_2  + B_2 \cos(2\pi\nu_2),
\]
where $A_2 = 1 - a^2 + b^2 - c^2$ and $B_2 = -2 (ac + b)$ if $g_2(\nu_2) > 0$ for all $\nu_2$;
otherwise, $A_2 = -(1 - a^2 + b^2 - c^2)$ and $B_2 = 2 (ac + b)$ if $g_2(\nu_2) < 0$ for all $\nu_2$.
In both cases, $A_2 = |g_2(\pi/2)| > 0$ and
\[
   A_2^2 - B_2^2 = (1 - a^2 + b^2 - c^2)^2 - 4 (ac + b)^2 = D > 0.
\]
According to Lemma~\ref{lemma:integration}, eq.~\eqref{eq:li-eq2},
\[
  \int_{-1/2}^{1/2} \frac{\exp(2\pi \mathrm{i} h_2 \nu_2)}{|g_2(\nu_2)|} \, d \nu_2
  =
  \frac{\beta^{|h_2|}}{\sqrt{A_2^2 - B_2^2}} = \frac{1}{\sqrt{D}}
\]
where
\[
  \beta = \frac{B_2}{A_2 + \sqrt{A_2^2 - B_2^2}} = \frac{B_2}{A_2 + \sqrt{D}} .
\]
The value of $\beta$ in
\begin{align*}
  \beta &= 0 \quad \mbox{if $ac + b = 0$}, \\
  \beta &= \frac{1 - a^2 + b^2 - c^2 - \sqrt{D}}{2 (ac+b)} \quad \mbox{if $ac+b \neq 0$ and $g_2(\nu_2)>0$ for all $\nu_2$}, \\
  \beta &= \frac{1 - a^2 + b^2 - c^2 + \sqrt{D}}{2 (ac+b)} \quad \mbox{if $ac+b \neq 0$ and $g_2(\nu_2)<0$ for all $\nu_2$};
\end{align*}
the formula that if valid in all cases is
\begin{equation}\label{eq:def-buni}
  \beta = \frac{2 (ac + b)}{1 - a^2 + b^2 - c^2 + \sign(1 - a^2 + b^2 - c^2) \sqrt{D}} .
\end{equation}

Finally,
\begin{align}
  \gamma_X(0,h_2) &= \int_{-1/2}^{1/2} \int_{-1/2}^{1/2} \frac{\sigma_\epsilon^2
  \exp(2 \pi \mathrm{i} h_2 \nu_2) }{g(\nu_1,\nu_2)}
  \, d\nu_1 \, d\nu_2
  = \nonumber \\ &=
  \int_{-1/2}^{1/2} \frac{\sigma_\epsilon^2
  \exp(2 \pi \mathrm{i} h_2 \nu_2) }{|g_2(\nu_2)|}
  \, d\nu_2
  = \frac{\beta^{|h_2|} \sigma_\epsilon^2}{\sqrt{D}} .
  \label{eq:gX0h2}
\end{align}
In particular,
\begin{equation}\label{eq:varXg00-g}
  \var X_{i,j} = \gamma_X(0, 0) = \frac{\sigma_\epsilon^2}{\sqrt{D}} .
\end{equation}

Due to symmetry,
\begin{equation}\label{eq:gXh10}
  \gamma_X(h_1,0) = \frac{\alpha^{|h_1|} \sigma_\epsilon^2}{\sqrt{D}},
\end{equation}
where
\begin{equation}\label{eq:def-auni}
  \alpha = \frac{2 (a + bc)}{1 + a^2 - b^2 - c^2 + \sign(1 + a^2 - b^2 - c^2) \sqrt{D}} .
\end{equation}

Notice that $|\alpha|<1$ and $|\beta|<1$, whence
\begin{equation}\label{neq:g00more}
  |\gamma_X(1,0)| < \gamma_X(0,0)
  \quad \mbox{and} \quad
  |\gamma_X(0,1)| < \gamma_X(0,0) .
\end{equation}
\subsection{Yule--Walker equations}
We formally write down Yule--Walker equations for the autocovariance function of $X$.
These equations are particular cases of ones given in \cite[Section~6]{Tjostheim:1978}:
\begin{align}
  \gamma_X(0,\,0) &= a \gamma_X(1,\,0) + b \gamma_X(0,\,1) + c\gamma_X(1,\,1) + \sigma_\epsilon^2,
    \label{eq:YW-0}\\
  \gamma_X(h_1,\,h_2) &= a \gamma_X(h_1-1,\, h_2) + b \gamma_X(h_1,\, h_2-1) + c \gamma_X(h_1-1,\, h_2-1) .
    \label{eq:YW-step}
\end{align}
In Lemma~\ref{lem:YW004} we obtain conditions for \eqref{eq:YW-step} to hold true.
As \eqref{eq:YW-0} requires specific conditions on coefficients $a$, $b$ and $c$,
we postpone the consideration of \eqref{eq:YW-0} to Section~\ref{sec:stability-causality}.
\begin{lemma}\label{lem:YW004}
  Let $X$ be a stationary field satisfying equation \eqref{eq:difeq1},
  and let $\gamma_X(h_1,h_2)$ be the covariance function of the process $X$.
  Then equality \eqref{eq:YW-step} holds true if in any of the following conditions
  hold true:
  \begin{enumerate}
    \def\labelenumi{(\arabic{enumi})}
    \item
      $h_1\ge 1$, {} $(1-a-b-c)(1+a-b+c)>0$ and $(1-a+b+c)(1+a+b-c)>0$,
    \item
      $h_1\le 0$, {} $(1-a-b-c)(1+a-b+c)<0$ and $(1-a+b+c)(1+a+b-c)<0$,
    \item
      $h_2\ge 1$, {} $(1-a-b-c)(1-a+b+c)>0$ and $(1+a-b+c)(1+a+b-c)>0$, or
    \item
      $h_2\le 0$, {} $(1-a-b-c)(1-a+b+c)<0$ and $(1+a-b+c)(1+a+b-c)<0$.
  \end{enumerate}
\end{lemma}
\begin{proof}
  According to Corollary~\ref{cor:hassectrdensity},
  the stationary field $X$ has a spectral density,
  that is \eqref{eq:def-sd-1} with $f_X(h_1,h_2)$ defined in \eqref{eq:sdensX}.
  Hence, as to rules how the Fourier transform changes when the function is linearly transformed,
  \begin{multline}\label{eq:ft-ltgamma}
    \gamma_X(h_1,\,h_2) - a \gamma_X(h_1-1, \, h_2) -
    b \gamma(h_1,\, h_2-1) - c \gamma(h_1-1,\, h_2-1)
    = \\
    \begin{aligned}[b]
      &=
      \int_{-1/2}^{1/2}
      \int_{-1/2}^{1/2}
      e^{2 \pi \mathrm{i} (h_1 \nu_1 + h_2 \nu_2)}
      (1 - a e^{- 2 \pi \mathrm{i} \nu_1}
         - b e^{- 2 \pi \mathrm{i} \nu_2}
         - c e^{- 2 \pi \mathrm{i} (\nu_1 + \nu_2)})
      \times \mbox{} \\ & \qquad \qquad \qquad \times
      f_X(h_1, h_2) \, d\nu_1 \, d\nu_2
      = \\ &=
      \int_{-1/2}^{1/2}
      \int_{-1/2}^{1/2}
      \frac{\exp(2\pi{\mathrm{i}} (h_1 \nu_1 + h_2 \nu_2)) \sigma_\epsilon^2}
      {1 - a e^{2 \pi \mathrm{i} \nu_1}
         - b e^{2 \pi \mathrm{i} \nu_2}
         - c e^{2 \pi \mathrm{i} (\nu_1 + \nu_2)}}
      \, d\nu_1 \, d\nu_2 .
    \end{aligned}
  \end{multline}

  \paragraph{Case (1):}
    $h_1\ge 1$, {} $(1-a-b-c)(1+a-b+c)>0$ and $(1-a+b+c)(1+a+b-c)>0$.
  In this case,
  \begin{gather*}
    (1 - b)^2 > (a + c)^2 \quad \mbox{and}  \quad (1+b)^2 > (a-c)^2, \\
    1 - a^2 + b^2 - c^2 > 2 b + 2 ac \quad \mbox{and} \quad 1 - a^2 + b^2 - c^2 > - 2 b - 2 a c, \\
    1 - a^2 + b^2 - c^2 > 2 \, | b + a c | .
  \end{gather*}
  Then, for every $\nu_2 \in [-1/2, \: 1/2]$, since $|\cos(2\pi \nu_2)| \le 1$,
  \begin{gather}
    1 - a^2 + b^2 - c^2 > 2 (b + a c) \cos(2 \pi \nu_2),
    \nonumber \\
    1 - 2 b \cos(2 \pi \nu_2) + b^2 > a^2 + 2 a c \cos(2 \pi \nu_2) + c^2,
    \nonumber \\
    |1 - b e^{2 \pi \mathrm{i} \nu_2}|^2 > |a + c e^{2 \pi \mathrm{i} \nu_2}|^2,
    \nonumber \\
    |1 - b e^{2 \pi \mathrm{i} \nu_2}| > |a + c e^{2 \pi \mathrm{i} \nu_2}|.
    \label{neq:cond_fpr_Lemma7}
  \end{gather}
  According to Lemma~\ref{lem:int-oneie},
  \begin{equation}\label{eq:1dinte0}
      \int_{-1/2}^{1/2}
      \frac{\exp(2\pi{\mathrm{i}} h_1 \nu_1)}
      {1 - a e^{2 \pi \mathrm{i} \nu_1}
         - b e^{2 \pi \mathrm{i} \nu_2}
         - c e^{2 \pi \mathrm{i} (\nu_1 + \nu_2)}}
      \, d\nu_1 = 0,
  \end{equation}
  which, with \eqref{eq:ft-ltgamma}, implies \eqref{eq:YW-step}.

  \paragraph{Case (2):}
    $h_1\le 0$, {} $(1-a-b-c)(1+a-b+c)<0$ and $(1-a+b+c)(1+a+b-c)<0$.
  In this case,
  \begin{gather*}
    (1 - b)^2 < (a + c)^2 \quad \mbox{and}  \quad (1+b)^2 < (a-c)^2, \\
    - 2 b - 2 ac < a^2 - b^2 + c^2 - 1 \quad \mbox{and} \quad 2 b + 2 a c < a^2 - b^2 + c^2 - 1, \\
    2 \, |b + a c| < a^2 - b^2 + c^2 - 1 .
  \end{gather*}
  Then, for every $\nu_2 \in [-1/2, \: 1/2]$, since $|\cos(2\pi \nu_2)| \le 1$,
  \begin{gather*}
    - 2 (b + a c) \cos(2 \pi \nu_2) \le 2 \,|b + ac| < a^2 - b^2 + c^2 - 1, \\
    1 - 2 b \cos(2 \pi \nu_2) + b^2 < a^2 + 2 a c \cos(2 \pi \nu_2) + c^2, \\
    |1 - b e^{2 \pi \mathrm{i} \nu_2}|^2 < |a + c e^{2 \pi \mathrm{i} \nu_2}|^2, \\
    |1 - b e^{2 \pi \mathrm{i} \nu_2}| < |a + c e^{2 \pi \mathrm{i} \nu_2}|.
  \end{gather*}
  Again, according to another case of Lemma~\ref{lem:int-oneie},
  \eqref{eq:1dinte0} still holds true.
  With \eqref{eq:ft-ltgamma}, it implies \eqref{eq:YW-step}.

 \paragraph{Cases (3) and (4)} are symmetric to cases (1) and (2), respectively.
\end{proof}

\subsection{Uniqueness of the solution}
We prove that under \textit{Sufficient\/} conditions of Proposition~\ref{prop:existance-eX-alt},
the stationary solution to \eqref{eq:difeq1} is unique.
\begin{proposition}\label{prop:uniquness-eX}
  Let $\sigma^2>0$, $a$, $b$ and $c$ be fixed real numbers,
  and $\epsilon_{i,j}$
  be a collection of uncorrelated random variables with zero mean and equal variance $\sigma^2$.
  If $D>0$, then equation \eqref{eq:difeq1}
  has only one stationary solution $X$,
  namely one defined by \eqref{eq:def-X-4ser} with coefficients $\psi_{k,l}$
  defined in \eqref{eq:psi-as-int}.
  The coefficients $\psi_{k,l}$ satisfy \eqref{eq:psi-summable-662}; hence, the double series in
  \eqref{eq:def-X-4ser} converges in least squares and almost surely.
\end{proposition}
\begin{proof}
  It remains to prove the uniqueness of the solution, as other assertions
  follow from the proof of Proposition~\ref{prop:existance-eX-alt}.

  Let $X$ be a solution to \eqref{eq:difeq1}, which can be rewritten as \eqref{eq:difeq1a}
  Let $\mathcal{H}$ be the Hilbert space spanned by $L_x^k L_y^l$, $k,l=\ldots,-1,0,1,\ldots$
  with scalar product
  \[
    \langle Y^{(1)}, Y^{(2)} \rangle = \ME Y^{(1)}_{0,0} Y^{(2)}_{0,0} .
  \]
  Formally, $\mathcal{H}$ can be defined as the closure of the set of random fields
  of form $\{\sum_{k=-n}^n \sum_{l=-n}^{n} \phi_{k,l} X_{i-k, j-l},\;
  \allowbreak i,j=\ldots,-1,0,1,\ldots\}$
  in a Banach space of bound\-ed-second-moment random fields with norm
  \[
    \|Y\| = \sup_{i,j} (\ME Y_{i,j}^2)^{1/2}.
  \]
  All random fields within $\mathcal{H}$ are jointly stationary, which imply that the scalar product indeed
  corresponds to the norm of the Banach space:
  \[
    \langle Y, Y\rangle = \|Y\|^2 \quad \mbox{for all} \quad Y\in\mathcal{H} .
  \]

  Since \eqref{eq:difeq1a}, $\epsilon \in \mathcal{H}$.
  Let us construct a similar space for $\epsilon$. Let
  \begin{align*}
    \mathcal{H}_{1} &= \left\{
      \sum_{k=-\infty}^\infty \sum_{l=-\infty}^\infty
      \theta_{k,l} L_x^k L_y^k \epsilon :
      \sum_{k=-\infty}^\infty \sum_{l=-\infty}^\infty \theta_{k,l}^2 < \infty
      \right\}
    = \\ &= \left\{
      \left\{
      \sum_{k=-\infty}^\infty \sum_{l=-\infty}^\infty
      \theta_{k,l} \epsilon_{i-k,j-l},\,
      i,j=\ldots,-1,0,1,\ldots\! \right\}
      :
      \sum_{k=-\infty}^\infty \sum_{l=-\infty}^\infty \theta_{k,l}^2 < \infty\!
      \right\}
  \end{align*}
  The series is convergent, see Proposition~\ref{prop:unconditional-double-series}.

  The random fields $\sigma_\epsilon^{-1} L_x^k L_y^k \epsilon$
  make an orthonormal basis of the subspace $\mathcal{H}_1$.
  The orthogonal projector onto $\mathcal{H}_1$ can be defined as follows
  \begin{gather*}
    P_1 Y = \frac{1}{\sigma_\epsilon^2} \sum_{k=-\infty}^\infty \sum_{l=-\infty}^\infty
    (\ME Y_{k,l} \epsilon_{0,0}) L^k L^l \epsilon,
    \\
    (P_1 Y)_{i,j} = \frac{1}{\sigma_\epsilon^2} \sum_{k=-\infty}^\infty \sum_{l=-\infty}^\infty
    \ME(Y_{k,l} \epsilon_{0,0}) \epsilon_{i-k,j-l}
  \end{gather*}
  for all $Y\in\mathcal{H}$.

  It is possible to verify that
  \[
    P_1 L_x = L_x P_1, \qquad
    P_1 L_y = L_y P_1, \qquad
    P_1 \epsilon = \epsilon .
  \]
  Applying the operator $P_1$ to both sides of \eqref{eq:difeq1a}, we get
  that the random field $P_1 X$ is also a stationary solution to \eqref{eq:difeq1}.
  Both random fields must be centered, and variances of both the fields satisfy \eqref{eq:varXg00-g}:
  \[
    \|X\|^2 = \|P_1 X\| = \frac{\sigma_\epsilon^2}{\sqrt{D}} .
  \]
  This implies that $\|X - P_1 X\| = 0$, which means $X = P_1 X$ almost surely.

  Combining \eqref{eq:psi-summable-662} and \eqref{eq:ft-ltgamma}, we get
  \begin{gather}
    \gamma_X(-k, -l) - a \gamma_X(-k{-}1, -l) -
    b \gamma_X(-k, -l{-}1) - c \gamma_X(-k{-}1, -l{-}1) = \sigma_\epsilon^2 \psi_{k,l}, 
    \nonumber \\
    \ME X_{0,0} X_{k,l} - a \ME X_{-1,0} X_{k,l} - b \ME X_{0,-1} X_{k,l} - c \ME X_{-1,-1}  X_{k,l} = \sigma_\epsilon^2 \psi_{k,l}, 
    \nonumber \\
    \ME (X_{0,0} - a X_{-1,0} - b X_{0,-1} - c X_{-1,-1}) X_{k,l} = \sigma_\epsilon^2 \psi_{k,l},
    \nonumber \\
    \ME \epsilon_{0,0} X_{k,l} = \sigma_\epsilon^2 \psi_{k,l},
    \label{eq:eXs2psi} \\
    X_{i,j} = P X_{i,j} = \frac{1}{\sigma_\epsilon^2}
            \sum_{k=-\infty}^\infty \sum_{l=-\infty}^\infty
            (\ME X_{k,l} \epsilon_{0,0}) \epsilon_{i-k,j-l}
            =
            \sum_{k=-\infty}^\infty \sum_{l=-\infty}^\infty
            \psi_{k,l} \epsilon_{i-k,j-l}.
    \nonumber
  \end{gather}
  Thus, the solution to \eqref{eq:difeq1} must satisfy \eqref{eq:def-X-4ser} almost surely,
  whence the uniqueness follows.
\end{proof}

\section{Stability and causality}
\label{sec:stability-causality}
\subsection{Causality}
We generalize the causality condition for random fields on a two-dimensional lattice.
\begin{definition}
  Let $\epsilon_{i,j}$  ($i$ and $j$ integers) be random variables
  of zero mean and constant, finite, non-zero variance.
  A random field $\{X_{i,j}, \; i,j=\ldots,-1,0,1,\ldots\}$
  is said to be \textit{causal} with respect to (w.r.t.\@)
  the white noise $\epsilon$ if
  there exists an array of coefficients $\{\psi_{k,l},\; k,l=0,1,2,\ldots\}$
  such that
  \begin{equation}\label{neq:causal_abscon}
    \sum_{k=0}^\infty \sum_{l=0}^\infty |\psi_{k,l}| < \infty
  \end{equation}
  and the process $X$
  allows the representation
  \begin{equation}\label{eq:causal_eq}
    X_{i,j} = \sum_{k=0}^\infty \sum_{l=0}^\infty \psi_{k,l} \epsilon_{i-k, j-l}.
  \end{equation}
\end{definition}
A causal random field must be wide-sense stationary.

\begin{proposition}\label{prop:withpsikl-causal}
  Let $\epsilon_{i,j}$
  be a uncorrelated zero-mean equal-variance variables,
  $\ME \epsilon_{i,j}=0$ and $\ME \epsilon_{i,j}^2 = \sigma_\epsilon^2$, {}
  $0 < \sigma_\epsilon^2 < \infty$
  Let $X$ be a stationary field that satisfies \eqref{eq:difeq1}.
  The random field $X$ is \textit{causal\/} w.r.t.\ the white noise $\epsilon$
  if and only if these four inequalities hold true:
  \begin{align*}
    1 - a - b - c &> 0, &  1 - a + b + c &> 0, \\
    1 + a - b + c &> 0, &  1 + a + b - c &> 0.
  \end{align*}
  If the field $X$ is causal, then the coefficients in representation \eqref{eq:causal_eq}
  equal
  \begin{equation}\label{eq:psikl-causal}
    \psi_{k,l}
    = \sum_{m=0}^{\min(k, l)}
      \binom{k}{m} \binom{l}{m}
      a^{k-m}  b^{l-m}  (ab + c)^m .
  \end{equation}
\end{proposition}

\begin{proof}
  \textit{Sufficiency.}
  According to Proposition~\ref{prop:uniquness-eX}, 
  the field $X$ admits a representation \eqref{eq:def-X-4ser}
  with coefficients $\psi_{k,l}$ satisfying \eqref{neq:causal_abscon}.
  Now prove that $\psi_{k,l} = 0$ if $k<0$ or $l<0$.
  
  Let $k<0$.
  Let us look into the proof of Lemma~\ref{lem:YW004},
  case (1), for $h_1 = -k$.
  Eq.~\eqref{eq:1dinte0} holds true,
  \[
      \int_{-1/2}^{1/2}
      \frac{\exp(-2\pi{\mathrm{i}} k \nu_1)}
      {1 - a e^{2 \pi \mathrm{i} \nu_1}
         - b e^{2 \pi \mathrm{i} \nu_2}
         - c e^{2 \pi \mathrm{i} (\nu_1 + \nu_2)}}
      \, d\nu_1 = 0,
  \]
  whence $\psi_{k,l}$ defined in \eqref{eq:psi-as-int} is zero,
  \[ 
    \psi_{k,l}
    =
      \int_{-1/2}^{1/2} e^{- 2 \pi \mathrm{i} l \nu_2}
      \biggl( \int_{-1/2}^{1/2}
      \frac{\exp(-2 \pi \mathrm{i} k \nu_1 )}
      {1 - a e^{2\pi \mathrm{i} \nu_1}
         - b e^{2\pi \mathrm{i} \nu_2} - c e^{2\pi \mathrm{i} (\nu_1+\nu_2)}}
      \, d\nu_1 \biggr) \,  d\nu_2 = 0 \,.
  \]

  The case $l<0$ is symmetric to $k<0$; if $l<0$, then  $\psi_{k,l} = 0$ also holds true. Thus,
  \eqref{eq:def-X-4ser} rewrites as \eqref{eq:causal_eq}.

  \textit{Necessity.}
  Denote
  \begin{equation}\label{eq:deff1..f4}
    f_1 \mathbin{=} 1 - a - b - c, \quad
    f_2 \mathbin{=} 1 - a + b + c, \quad
    f_3 \mathbin{=} 1 + a - b + c, \quad
    f_4 \mathbin{=} 1 + a + b - c.
  \end{equation}
  Obviously, $f_1 + f_2 + f_3 + f_4 = 4 > 0$,
  and due to Proposition~\ref{prop:existance-eX-newps},
  $D = f_1 f_2 f_3 f_4 > 0$.

  Now prove that $f_1 f_3 > 0$. Equality $f_1 f_3 = 0$ is impossible. Let $f_1 f_3 < 0$.
  Then $f_2 f_4 < 0$.
  The representation \eqref{eq:def-X-4ser} is unique: the coefficients $\psi_{k,l}$ must be equal
  to $\sigma_\epsilon^{-2} \cov(X_{0,0},\epsilon_{-k,-l})$.
  Thus, coefficients in \eqref{eq:causal_eq} must be defined by \eqref{eq:psi-as-int}.
  Prove that $\psi_{k,l}=0$ for all integers $k>0$ and $l>0$.
  Look into the proof of Lemma~\ref{lem:YW004},
  case (2), for $h_1 = -k$. Eq.~\eqref{eq:1dinte0} still holds true
  and implies $\psi_{k,l}=0$ the same way as in the proof of \textit{sufficiency}.
  The representation \eqref{eq:causal_eq} rewrites as $X_{i,j} = 0$.
  The assumption $f_1 f_3 < 0$ have led to a contradiction.
  Thus, $f_1 f_3 > 0$.

  The inequality $f_1 f_2 > 0$ holds true due to symmetry.
  Inequalities $f_1 f_2 > 0$, $f_1 f_3 > 0$, $f_1 f_2 f_3 f_4> 0$ and $f_1 + f_2 + f_3 + f_4 > 0$
  imply that $f>0$, $f_2>0$, $f_3>0$ and $f_4>0$.

  \textit{The expression for $\psi_{k,l}$.}
  Again, if $f_1>0$, $f_2>0$, $f_3>0$ and $f_4>0$,
  in the proof of Lemma~\ref{lem:YW004}, case (1), inequality \eqref{neq:cond_fpr_Lemma7} was demonstrated. 
  Let $k$ and $l$ be nonnegative integers.
  By \eqref{eq:psi-as-int} and Lemmas \ref{lem:int-oneie} and \ref{lem:int-oneie2},
  \begin{align*}
    \psi_{k,l}
    &=
      \int_{-1/2}^{1/2} e^{- 2 \pi \mathrm{i} l \nu_2}
      \biggl( \int_{-1/2}^{1/2}
      \frac{\exp(-2 \pi \mathrm{i} k \nu_1 )}
      {1 - a e^{2\pi \mathrm{i} \nu_1}
         - b e^{2\pi \mathrm{i} \nu_2} - c e^{2\pi \mathrm{i} (\nu_1+\nu_2)}}
      \, d\nu_1 \biggr) \,  d\nu_2  
    = \\ &=
      \int_{-1/2}^{1/2}
      \frac{ e^{- 2 \pi \mathrm{i} l \nu_2} (a + c e^{2 \pi \mathrm{i} \nu_2})^k }{(1 - b e^{2\pi \mathrm{i} \nu_2})^{k+1}}
      \, d\nu_2
    = \\ &=
      \sum_{n=0}^{\min(k,l)} \binom{k}{n} \binom{l}{n} a^{k-n} b^{l-n} (ab + c)^n .
      \qedhere
  \end{align*}
\end{proof}

\subsection{Autocovariance function and Yule--Walker equations under the causality condition}
\begin{proposition}\label{prop:YW-nscond}
  Let coefficients $a$, $b$ and $c$ be such that $f_1>0$, $f_2>0$, $f_3>0$ and $f_4>0$
  for $f_1,\ldots,f_4$ defined in \eqref{eq:deff1..f4}.
  Let $X$ be a stationary field and $\epsilon$ be a collection of zero-mean equal-variance
  random variables, $\var \epsilon_{i,j} =\sigma_\epsilon^2$,  that satisfy \eqref{eq:difeq1}.
  Then $\cov(X_{i,j}, \epsilon_{k,l}) = 0$ if $k>i$ or $l > j$, and the following Yule--Walker
  equations hold true:
  \begin{align*}
    \gamma_X(0,0) &=  a \gamma_X(1,0) + b \gamma_X(0,1) + c \gamma_X(1,1) + \sigma_\epsilon^2, \\
    \gamma_X(h_1,h_2) &=  a \gamma_X(h_1{-}1,h_2) + b \gamma_X(h_1,h_2{-}1) + c \gamma_X(h_1{-}1,h_2{-}1)
  \end{align*}
  if $\max(h_1,h_2) > 0$.
\end{proposition}

\begin{proof}
  Because of causality and representation \eqref{eq:causal_eq},
  \[
    \cov(X_{i,j},\epsilon_{k,l}) = \sum_{s=0}^\infty \sum_{r=0}^\infty \psi_{r,s} \cov(\epsilon_{i-r,j-s} \epsilon_{k,l}) .
  \]
  If $k>i$, then $i-r<k$ in all summands; if $l>j$, then $j-s<l$.
  In ether case, $\cov(\epsilon_{i-r,j-s} \epsilon_{k,l}) = 0$ is all terms. Hence,
  $\cov(X_{i,j},\epsilon_{k,l}) = 0$.

  Eq.~\eqref{eq:YW-step} for $h_1>0$ or $h_2>0$ follows from Lemma~\ref{lem:YW004}, case (1) and case(3), respectively.

  Because of \eqref{eq:difeq1},
  \begin{gather*}
    \cov(X_{i,j}, \epsilon_{i,j}) = \cov(a X_{i-1,j} + b X_{i,j-1} + c X_{i-1,j-1} + \epsilon_{i,j}, \: \epsilon_{i,j}), \\
    \begin{aligned}
      \cov(X_{i,j}, \epsilon_{i,j}) &= a \cov(X_{i-1,j}, \epsilon_{i,j}) + b \cov(X_{i,j-1},\epsilon_{i,j})
       + \mbox{} \\ & \quad +
       c \cov(X_{i-1,j-1},\epsilon_{i-1,j-1}) + \cov(\epsilon_{i,j}, \epsilon_{i,j}),
    \end{aligned} \\
    \cov(X_{i,j}, \epsilon_{i,j}) = \sigma_\epsilon^2. 
  \end{gather*}
  Then
  \begin{gather*}
    \cov(X_{i,j}, X_{i,j}) = \cov(X_{i,j}, \: a X_{i-1,j} + b X_{i,j-1} + c X_{i-1,j-1} + \epsilon_{i,j}), \\
    \begin{aligned}
      \cov(X_{i,j}, X_{i,j}) &= a \cov(X_{i,j},X_{i-1,j}) + b \cov(X_{i,j},X_{i,j-1})
      + \mbox{} \\ & \quad +
      c \cov(X_{i,j} X_{i-1,j-1}) +
        \cov(X_{i,j},\epsilon_{i,j}),
    \end{aligned} \\
    \gamma_X(0, 0) = a \gamma_X(1, 0) + b \gamma_X(0, 1) + c \gamma_X(1, 1) + \sigma_\epsilon^2 .
    \qedhere
  \end{gather*}
\end{proof}

\begin{proposition}\label{prop:covXe-cc}
  Let coefficients $a$, $b$ and $c$ be such that $f_1>0,\:\ldots,\: f_4>0$,
  with $f_1,\ldots,f_4$ defined in \eqref{eq:deff1..f4}.
  Let $X$ be a stationary random field, and $\epsilon$ be
  a collection of zero-mean random variables with equal variance
  $\ME \epsilon_{i,j}^2$
  that satisfy \eqref{eq:difeq1}.
  Then random fields $X$ and $\epsilon$ are jointly stationary
  with cross-autocovariance function
  \begin{gather*}
    \cov(X_{i+h_1,j+h_2}, \epsilon_{i,j}) = 0 \qquad
    \mbox{if $h_1 < 0$ or $h_2 < 0$}, \\
    \cov(X_{i+h_1,j+h_2}, \epsilon_{i,j})
    = \sum_{k=0}^{\min(h_1,h_2)}
      \binom{h_1}{k} \binom{h_2}{k}
      a^{h_1-k}  b^{h_2-k}   (ab + c)^k \sigma_\epsilon^2
  \end{gather*}
  if $h_1\ge 0$ and $h_2 \ge 0$.
\end{proposition}
\begin{proof}
  The joint stationarity follows from \eqref{eq:difeq1a} and the stationarity of $X$
  (this was already used in the proof of Proposition~\ref{prop:uniquness-eX}).
  The formula for the autocovariance function follows from
  \eqref{eq:eXs2psi} and Proposition~\ref{prop:withpsikl-causal}.
\end{proof}

\begin{proposition}\label{prop:gh1h2_ads}
  Under conditions of Proposition~\ref{prop:covXe-cc},
  the autocovariance function of the field $X$
  satisfies
  \begin{equation}\label{eq:acovf-X-multiplicative}
    \gamma_X(h_1,h_2) = \frac{|\alpha|^{|h_1|} |\beta|^{|h_2|} \sigma_\epsilon^2}{\sqrt{D}}
    \qquad \mbox{if $h_1 h_2 \le 0$,}
  \end{equation}
  where $D$ is defined in Lemma~\ref{lemma:denominator}, and
  \[
  \alpha = \frac{2 (a + bc)}{1 + a^2 - b^2 - c^2 + \sqrt{D}} , \qquad
  \beta = \frac{2 (ac + b)}{1 - a^2 + b^2 - c^2 + \sqrt{D}} .
  \]
\end{proposition}

\begin{proof}
  First, notice that
  \begin{gather*}
    1 + a^2 - b^2 - c^2 = \frac{(1{-}a{-}b{-}c)(1{-}a{+}b{+}c) + (1{+}a{-}b{+}c)(1{+}a{+}b{-}c)}{2} > 0, \\
    1 - a^2 + b^2 - c^2 = \frac{(1{-}a{-}b{-}c)(1{+}a{-}b{+}c) + (1{-}a{+}b{+}c)(1{+}a{+}b{-}c)}{2} > 0.
  \end{gather*}
  Thus, $\alpha$ and $\beta$ are the same as in \eqref{eq:def-buni} and \eqref{eq:def-auni},
  and equality in \eqref{eq:acovf-X-multiplicative}
  holds true for $h_1 h_2 = 0$.

  Consider three cases.
  \paragraph{Case 1:} $b\neq 0$.
  Eq.~\eqref{eq:acovf-X-multiplicative} hold true for $h_1 = 0$ and all $h_2$.
  Assume it holds for $h_1 = n_1$ and all $h_2$, for $n_1$ a nonnegative integer.
  Eq.~\eqref{eq:acovf-X-multiplicative} hold true for $h_2 = 0$ (and all $h_1$).
  Assume it holds true for $h_2 = -n_2$ and $h_1 = n_1+1$, for $n_2$ a nonnegative integer.
  Now prove that \eqref{eq:acovf-X-multiplicative} holds true for $h_2 = -n_2-1$ and all $h_1 = n_1+1$.
  Apply Lemma~\ref{lem:YW004}, conditions (1), for $h_1 = n_1+1\ge 1$ and $h_2 = -n_2$:
  \[
    \gamma_X(n_1{+}1,\,{-}n_2) = a \gamma_X(n_1,\, -n_2) + b \gamma_X(n_1{+}1,\, -n_2{-}1) + c \gamma_X(n_1,\, -n_2-1).
  \]
  Hence
  \begin{align*}
    \gamma_X(n_1{+}1,{-}n_2{-}1)
    &=
    \frac{\gamma_X(n_1{+}1,\,-n_2) - a \gamma_X(n_1,\, -n_2) - c \gamma_X(n_1,\, -n_2-1)} {b}
    = \\ &=
    \frac{1}{b} \left(\frac{\alpha^{n_1+1} \beta^{n_2} \sigma_\epsilon^2}{\sqrt{D}} -
    \frac{a \alpha^{n_1} \beta^{n_2} \sigma_\epsilon^2}{\sqrt{D}} -
    \frac{c \alpha^{n_1} \beta^{n_2+1} \sigma_\epsilon^2}{\sqrt{D}} \right)
    = \\ &=
    \frac{\alpha^{n_1} \beta^{n_2} \sigma_\epsilon^2 \, (\alpha - a - c \beta)}{b \, \sqrt{D}}
    =
    \frac{\alpha^{n_1+1} \beta^{n_2+1} \sigma_\epsilon^2}{\sqrt{D}};
  \end{align*}
  equality $\alpha - a - c \beta = b \alpha \beta$ can be verified by easy but tedious computations.
  Thus, \eqref{eq:acovf-X-multiplicative} holds true for $h_2 = -n_2-1$ and all $h_1 = n_1+1$.
  By induction, it holds true for $h_1 = n_1+1$ and all $h_2 \le 0$,
  and then for all $h_1\ge 0$ and $h_2 \le 2$.

  Since both sides of \eqref{eq:acovf-X-multiplicative} are even functions in $(h_1,h_2)$,
  that is $\gamma_X(h_1,h_2) = -\gamma_X(-h_1, -h_2)$,
  \eqref{eq:acovf-X-multiplicative} also holds true for $h_1 \le 0$ and $h_2 \ge 0$.
  Thus, \eqref{eq:acovf-X-multiplicative} holds true for all integer $h_1$ and $h_2$
  such that $h_1 h_2 \le 0$.

  \paragraph{Case 2:} $b = 0$ but $a \neq 0$.
  This case is symmetrical to case $a = 0$ and $b\neq 0$, which is covered by case 1.

  \paragraph{Case 3:} $a = b = 0$.  Then $\alpha = \beta = 0$ and
  \[
    \gamma_X(h_1,h_2) = \int_{-1/2}^{1/2} \int_{-1/2}^{1/2}
      \frac{\exp(2 \pi \mathrm{i} (h_1 \nu_1 + h_2 \nu_2))}
      {|1 - c \exp(2 \pi \mathrm{i} (\nu_1 + \nu_2))|^2}
      \, d\nu_1 \, d\nu_2 .
  \]
  Substituting $\nu_1 = \nu_3 -\nu_2$ and taking in account the periodicity, we get
  \begin{align*}
    \gamma_X(h_1,h_2)
      &=
      \int_{\nu_2 = -1/2}^{1/2}  \int_{\nu_3=\nu_2 - \frac12}^{\nu_2 + \frac12}
      \frac{\exp(2 \pi \mathrm{i} (h_1 \nu_3 - h_1 \nu_2 + h_2 \nu_2))}
      {|1 - c \exp(2 \pi \mathrm{i} (\nu_3))|^2}
      \, d\nu_3 \, d\nu_2
      = \\ &=
      \int_{\nu_2 = -1/2}^{1/2}  \int_{\nu_3= - \frac12}^{\frac12}
      \frac{\exp(2 \pi \mathrm{i} (h_1 \nu_3 - h_1 \nu_2 + h_2 \nu_2))}
      {|1 - c \exp(2 \pi \mathrm{i} \nu_3)|^2}
      \, d\nu_3 \, d\nu_2
      = \\ &=
      \int_{-1/2}^{1/2} \exp(2 \pi \mathrm{i} (h_2 - h_1) \nu_2) \, d\nu_2 \,
      \int_{\nu_3=-\frac12}^{\frac12}
      \frac{\exp(2 \pi \mathrm{i} h_1 \nu_3)}
      {1 - 2 c \cos (2 \pi \nu_3) + c^2 }
      \, d\nu_3 .
  \end{align*}
  If $h_1 h_2 <0$, then $h_1 \neq h_2$ and
      $\int_{-1/2}^{1/2} \exp(2 \pi \mathrm{i} (h_2 - h_1) \nu_2) \, d\nu_2 = 0$,
  whence $\gamma_X(h_1, h_2) = 0$, and \eqref{eq:acovf-X-multiplicative}
  holds true.
  As noted before \eqref{eq:acovf-X-multiplicative} holds true if $h_1 h_2 = 0$.
  Thus, \eqref{eq:acovf-X-multiplicative} holds true for all integer $h_1$ and $h_2$ such that
  $h_1 h_2 \le 0$.
\end{proof}
Under conditions of Proposition~\ref{prop:gh1h2_ads},
\begin{equation}\label{eq:g1-1}
  \gamma_X(1,-1) \gamma_X(0,0) = \gamma_X(1,0) \gamma_X(0,1) .
\end{equation}

\subsection{Stability}
\begin{lemma}\label{lem:psideq}
  Let $a$, $b$ and $c$ be such that $f_1,\ldots,f_4$ defined in \eqref{eq:deff1..f4}
  are positive real numbers.
  Let $\psi_{k,l}$ be defined in \eqref{eq:psikl-causal} if $k\ge 0$ and $l\ge 0$,
  and $\psi_{k,l} = 0$ if $k<0$ or $l<0$.
  Then $\psi_{k,l}$ satisfy this equation:
  \begin{equation}\label{eq:psideq}
    \psi_{k,l} - a \psi_{k-1,l} - b \psi_{k,l-1} - c \psi_{k-1,l-1} =
    \begin{cases}
      1 \quad \mbox{if $k=l=0$},\\
      0 \quad \mbox{otherwise}.
    \end{cases}
  \end{equation}.
\end{lemma}
\begin{proof}
  Under conditions $f_1>0,\,\ldots,\,f_4>0$,
  the coefficients $\psi_{k,l}$ satisfy \eqref{eq:psi-as-int}
  for all integer $k$ and $l$.
  Respective transforms of the Fourier coefficients and the integrand in
  \eqref{eq:psi-as-int} yield
  \begin{multline*}
    \psi_{k,l} - a \psi_{k-1,l} - b \psi_{k,l-1} - c \psi_{k-1,l-1}
    = \\ =
      \int_{-1/2}^{1/2} \int_{-1/2}^{1/2}
      \exp(-2 \pi \mathrm{i} (k \nu_1 + l \nu_2))
      \, d\nu_1 \, d\nu_2
    =
    \begin{cases}
      1 \quad \mbox{if $k=l=0$}, \\
      0 \quad \mbox{otherwise}.
    \end{cases}
     \qedhere
  \end{multline*}
\end{proof}
\begin{remark}
  In Lemma~\ref{lem:psideq} the condition on $a$, $b$ and $c$
  is unnecessary.
  Indeed, 
  the coefficients $\psi_{k,l}$ are polynomials in $a$, $b$ and $c$;
  thus, the left-hand side of \eqref{eq:psideq} is also a polynomial in $a$, $b$ and $c$
  and the fact that \eqref{eq:psideq} holds true for small $a$, $b$ and $c$ implies
  that \eqref{eq:psideq} holds true for all real $a$, $b$ and $c$.
\end{remark}

\begin{proposition}
  Let $a$, $b$ and $c$ be such that $f_1,\ldots,f_4$ defined in \eqref{eq:deff1..f4}
  are positive real numbers.
  Consider the recurrence equation
  \begin{equation}\label{eq:xij-dererministic}
    x_{i,j} = a x_{i-1,j} + b x_{i,j-1} + c x_{i-1,j-1} + v_{i,j},
  \end{equation}
  for all positive integers $i$ and $j$,
  where
    \begin{itemize}
        \item  $v_{i,j}$ are known variables, $i>0$ and $j>0$;
        \item  $x_{0,0}$, $x_{i,0}$ and $x_{0,j}$ are preset,
          so setting their values makes initial/boundary conditions;
        \item  $x_{i,j}$ ($i>0$ and $j>0$) are unknown variables.
    \end{itemize}
  Then the explicit formula for the solution is the following:
  \begin{align}
    x_{i,j} &= \psi_{i,j} x_{0,0} + \sum_{k=0}^{i-1} \psi_{k,j} \, (x_{i-k,0} - a x_{i-k-1,0})
    + \mbox{} \nonumber \\ &\quad +
    \sum_{l=0}^{j-1} \psi_{i,l} \, (x_{0,j-l} - b x_{0,j-l-1}) 
    + \sum_{k=0}^{i-1} \sum_{l=0}^{j-1} \psi_{k,l} v_{i-k,j-l},
    \label{eq:xij-explbosol}
  \end{align}
  with $\psi_{k,l}$ defined in \eqref{eq:psikl-causal}.

  If, in addition, the sequences of $x_{i,0}$ and $x_{0,j}$ and the array of $v_{i,j}$
  are bounded, then
  the solution $x$ is also bounded.
\end{proposition}

\begin{proof}
  First, show that under convention that ``an empty sum is zero'', namely
  $\sum_{i=0}^{-1} \ldots = 0$ and $\sum_{j=0}^{-1} \ldots = 0$,
  the right side of \eqref{eq:xij-explbosol} satisfies the bounding conditions.
  For one or both indices being zero, the coefficients $\psi_{k,l}$ equal
  \[
    \psi_{0,0} = 1, \qquad \psi_{k,0} = a^k, \qquad \psi_{0,l} = b^l
  \]
  for all integer $k\ge 0$ and $l\ge 0$. Hence,
  \begin{align*}
    \psi_{0,0} x_{0,0} &= x_{i,j}, \\
    \psi_{0,j} x_{0,0} &+ \sum_{l=0}^{j-1} \psi_{0,l} (x_{0,j-l} - b x_{0,j-l-1})
    =
    b^l x_{0,0} + \sum_{l=0}^{j-1} b^l (x_{0,j-k} - b x_{0,j-l-1})
    = \\ &=
    b^l x_{0,0} + \sum_{l=0}^{j-1} b^l x_{0,j-l} - \sum_{l=0}^{j-1} b^{l+1} x_{0,j-l-1}
    = \\ &= 
    b^l x_{0,0} + \sum_{l=0}^{j-1} b^l x_{0,j-l} - \sum_{l=1}^{j} b^{l} x_{0,j-l}
    = 
    b^l x_{0,0} + x_{0,j} - b^l x_{0,0} = x_{0,j}
  \end{align*}
  for integer $j>0$, and due to symmetry for all integers $i>0$
  \[
    \psi_{i,0} x_{0,0} + \sum_{k=0}^{i-1} \psi_{k,0}\, (x_{i-k,0} - a x_{i-k-1,0} ) = x_{i,0} .
  \]

  Known variables $v_{i,j}$ are defined for all integer $i>0$ and $j>0$. Extend their domain
  for all integer $i\ge 0$ and $j\ge 0$ by assigning
  \[
    v_{0,0} = 0; \qquad
    v_{i,0} = x_{i,0} - a x_{i-1,0}, \quad i>0; \qquad
    v_{0,j} = x_{0,j} - b x_{0,j-1}, \quad j>0.
  \]
  Then the right-hand side of \eqref{eq:xij-explbosol} is
  \[
    x^{\rm RHS}_{i,j} = \sum_{k=0}^i \sum_{l=0}^j \psi_{k,l} v_{i-k,j-l}, \qquad
    i\ge 0, \quad j\ge 0.
  \]
  In particular, $x^{\rm RHS}_{i,j} = x_{i,j}$ if either $i=0$ or $j=0$.

  Now prove that $x^{\rm RHS}$ satisfies \eqref{eq:xij-dererministic}.
  To that and, let $\psi_{k,l} = 0$ if $k<0$ or $l<0$.
  Then, for $i$ and $j$ positive integers,
  \begin{multline*}
    x^{\rm RHS}_{i,j} - a x^{\rm RHS}_{i-1,j} - b x^{\rm RHS}_{i,j-1} - c x^{\rm RHS}_{i-1,j-1}
    = \\
    \begin{aligned}
      &=
      \sum_{k=0}^i \sum_{l=0}^j \psi_{k,l} v_{i-k,j-l}
      - a \sum_{k=0}^{i-1} \sum_{l=0}^j \psi_{k,l} v_{i-1-k,j-l}
      - \mbox{} \\ & \quad -
      b \sum_{k=0}^i \sum_{l=0}^{j-1} \psi_{k,l} v_{i-k,j-1-l}
      - c \sum_{k=0}^{i-1} \sum_{l=0}^{j-1} \psi_{k,l} v_{i-1-k,j-1-l}
      = \\ &=
      \sum_{k=0}^i \sum_{l=0}^j \psi_{k,l} v_{i-k,j-l}
      - a \sum_{k=1}^i \sum_{l=0}^j \psi_{k-1,l} v_{i-k,j-l}
      - \mbox{} \\ & \quad -
      b \sum_{k=0}^i \sum_{l=1}^j \psi_{k,l-1} v_{i-k,j-l}
      - c \sum_{k=1}^i \sum_{l=1}^j \psi_{k-1,l-1} v_{i-k,j-l}
      = \\ &=
      \sum_{k=0}^i \sum_{l=0}^j (\psi_{k,l} - a \psi_{k-1,l} - b \psi_{k,l-1} - c \psi_{k-1,l-1}) v_{i-k,j-l}
      =
      v_{i,j}\, ;
    \end{aligned}
  \end{multline*}
  here we used Lemma~\ref{lem:psideq}.

  Finally, equality $x_{i,j} = x^{\rm RHS}_{i,j}$ for all integer $i\ge 0$ and $j\ge 0$
  can be proved by induction.

  If initial/boundary values and $v_{i,j}$ are bounded,
  \[
    \sup_{i>0} |x_{i,0}|< \infty, \qquad
    \sup_{j>0} |x_{0,j}|< \infty, \quad \mbox{and} \quad
    \sup_{i,j>00} |v_{i,j}| < \infty,
  \]
  then
  \begin{align*}
    |x_{i,j}| &= \left|\sum_{k=0}^i \sum_{k=0}^j \psi_{k,l} v_{i-k,j-l}\right|
    \le \sum_{k=0}^i \sum_{j=0}^l |\psi_{i,j}| \max_{0\le r \le 1} \max{0\le s \le 1} |v_{r,s}|
    \le \\ &\le
    \sum_{k=0}^i \sum_{j=0}^l |\psi_{i,j}|
    \times \begin{aligned}[t] \max\Bigl(
      &|x_{0,0}|, \; (1+|a|) \max_{1\le r \le i} |x_{r,0}|, \\
      &(1+|b|) \max_{1\le s \le i} |x_{0,s}|, \;
      \max_{1\le r \le i} \max_{1\le s \le i} |v_{i,j}| \Bigr) \le
    \end{aligned}
    \\ &\le
    \sum_{k=0}^{\infty} \sum_{j=0}^{\infty} |\psi_{i,j}|
    \times \begin{aligned}[t] \max\Bigl(
      &|x_{0,0}|, \; (1+|a|) \sup_{r \ge 1} |x_{r,0}|, \\
      &(1+|b|) \sup_{s \ge 1} |x_{0,s}|, \; 
      \sup_{r, s \ge 1} |v_{i,j}| \Bigr) < \infty .
    \end{aligned}
  \end{align*}
  Thus, the solution $x$ is bounded.  Here we used \eqref{neq:causal_abscon}.
\end{proof}

\section{Symmetry, nonidentifiability, and special cases}
\label{sec:symmetry}
In this section we consider the question how parameters change when the
random field $X$ is flipped, and whether two or more sets of parameters
correspond to the same distribution (or the same autocovariance function) of $X$.
Here parameters are 
the coefficients $a$, $b$ and $c$ and the variance of
the error term $\sigma_\epsilon^2$.
Usually, two sets sets of parameters correspond to the same autocovariance function.
In special cases, up to four 
sets sets of parameters correspond to the same autocovariance function.
We start with special cases.

\subsection{Special cases}
\subsubsection{Symmetric case: $a b + c = 0$.}
\label{sss:symmetric_case}

\begin{proposition}
  Let 
  a stationary field $X$ and
  a collection of uncorrelated zero-mean equal-variance random variables $\epsilon$
  be a solution to \eqref{eq:difeq1}.
  The autocovariance function $\gamma_X$ is even with respect to each variable,
  $\gamma_X(h_1, h_2) = \gamma_X(-h_1, h_2) = \gamma_X(h_1, -h_2)$
  if and only if $ab + c = 0$.
\end{proposition}
\begin{proof}
  Eq.~\eqref{eq:def-sd-1} can be rewritten as
\begin{equation*}
  \gamma_X(h_1,\,h_2)
  =
  \int_{-1/2}^{1/2} e^{2 \pi {\mathrm{i}} h_2 \nu_2}
  \int_{-1/2}^{1/2} e^{2 \pi {\mathrm{i}} h_1 \nu_1}
  f_X(\nu_1,\nu_2) \, d\nu_1 \, d\nu_2 .
\end{equation*}
The even symmetry with respect to $h_2$, $\gamma_X(h_1,h_2) = \gamma_X(h_1, -h_2)$,
is equivalent to the inner integral being a real function,
\begin{align*}
  & \int_{-1/2}^{1/2} e^{2 \pi {\mathrm{i}} h_1 \nu_1}
  f_X(\nu_1,\nu_2) \, d\nu_1
  = \\ &=
  \int_{-1/2}^{1/2}
  \frac {\exp(2 \pi \, {\mathrm{i}} \, h_1 \nu_1) \sigma_\epsilon^2}
  {|1 - a e^{2 \pi {\mathrm{i}} \nu_1}
      - b e^{2 \pi {\mathrm{i}} \nu_2} - c e^{2 \pi {\mathrm{i}} (\nu_1 + \nu_2)} |^2}
  \, d\nu_1
  = \\ &=
  \int_{-1/2}^{1/2}
  \frac {\exp(2 \pi \, {\mathrm{i}} \, h_1 \nu_1) \sigma_\epsilon^2}
  {A + B \cos(2 \pi \nu_1) + C \sin(2 \pi \nu_2)}
  \in \mathbb{R} \mbox{\quad for almost all $\nu_2\in\left[-\frac12, \frac12\right]$}
  ,
\end{align*}
with $A$, $B$ and $C$ defined in \eqref{eq:ABC-denom} in Section~\ref{ss:h1eq0}.
In turn, this is equivalent to the function Fourier-transformed
  being an even function,
\[
  \frac{\sigma_\epsilon^2}
  {A + B \cos(2 \pi \nu_1) + C \sin(2 \pi \nu_1)}
  =
  \frac {\sigma_\epsilon^2}
  {A + B \cos(-2 \pi \nu_1) + C \sin(-2 \pi \nu_1)}
\]
for almost all $\nu_1$ and $\nu_2\in\left[-\frac12, \frac12\right]$,
which holds true if and only if
$C = 2 (ab + c) \sin(2 \pi \nu_2) = 0$ for all $\nu_2$,
or $ab + c = 0$.
\end{proof}

Thus, we consider the case $c = - ab$.
The following proposition establishes conditions on $a$ and $b$.

\begin{proposition}\label{prop:scsym-acf}
  Let $c = -ab$. The stationary field $X$
  and a collection of uncorrelated zero-mean variables with equal variance $\sigma_\epsilon^2$
  which satisfy \eqref{eq:difeq1}
  exist if and only if $|a| \neq 1$ and $|b| \neq 1$.
  In that case, the autocovariance function of $X$ is equal
  \begin{equation}\label{eq:gamma-symmetric}
    \gamma_X(h_1,h_2) = \frac{a^{\pm h_1} b^{\pm h_2} \sigma_\epsilon^2}{|1 - a^2| \times |1 - b^2|},
  \end{equation}
  where the signs ``$\pm$'' are chosen so that $|a^{\pm h_2}|<1$ and $|b^{\pm h_2}|<1$.

  The field $X$ is causal with respect to $\epsilon$ is and only if $|a|<1$ and $|b|<1$.
\end{proposition}
\begin{proof}
  Apply Proposition~\ref{prop:existance-eX-newps}.
  Since $c=-ab$,
  \begin{align*}
    D &= (1 - a - b + a b) (1 - a + b - a b) (1 + a - b - a b) (1 + a + b + a b)
      = \\ &=
         (1 - a)^2 (1 - b)^2 (1 + b)^2 (1 + a)^2 .
  \end{align*}
  Thus, $D\ge 0$, and the necessary and sufficient condition $D>0$
  for the desired process to exist is satisfied if and only if
  $(1-a) (1-b) (1+b) (1+a) \neq 0$,
  which is equivalent to $|a|\neq 1$ and $|b| \neq 1$.

  The spectral density \eqref{eq:sdensX} splits into two factors
  \[
    f_X(\nu_1,\nu_2) = \frac{1}{|1 - a e^{2 \pi \mathrm{i} \nu_1}|^2} \times
    \frac{\sigma_\epsilon^2}{|1 - b e^{2 \pi \mathrm{i} \nu_2}|^2} .
  \]
  Thus, the autocovariance function is
  \[
    \gamma_X(h_1,\,h_2) =
    \int_{-1/2}^{1/2} 
      \frac{\exp(2 \pi h_1 \nu_1)}{|1 - a e^{2\pi \mathrm{i} \nu_1}|^2} \, d \nu_1 \times
    \int_{-1/2}^{1/2} 
      \frac{\exp(2 \pi h_2 \nu_2)}{|1 - b e^{2\pi \mathrm{i} \nu_2}|^2} \, d \nu_2 \, \sigma_\epsilon^2.
  \]
  With Lemma~\ref{lemma:integration},
  \begin{align*}
    \int_{-1/2}^{1/2} 
      \frac{\exp(2 \pi h_1 \nu_1)}{|1 - a e^{2\pi \mathrm{i} \nu_1}|^2} \, d \nu_1
    &=
    \int_{-1/2}^{1/2} 
      \frac{\exp(2 \pi h_1 \nu_1)}{1 + a^2 - 2 a \cos (2\pi  \nu_1)} \, d \nu_1
    = \\ &=
    \frac{\alpha^{|h_1|}}{\sqrt{(1+a^2)^2 - 4a^2}}
    =
    \frac{\alpha^{|h_1|}}{|1 - a^2|}
  \end{align*}
  with
  \[
    \alpha = \frac{2 a}{1 + a^2 + |1 - a^2|} = \begin{cases}
      a & \mbox{if $|a|<1$}, \\
      a^{-1} & \mbox{if $|a|>1$} .
    \end{cases}
  \]
  Thus,
  \[
    \int_{-1/2}^{1/2} 
      \frac{\exp(2 \pi h_1 \nu_1)}{|1 - a e^{2\pi \mathrm{i} \nu_1}|^2} \, d \nu_1
    =
    \frac{a^{\pm h_1}}{|1 - a^2|} .
  \]
  Similarly,
  \[
    \int_{-1/2}^{1/2} 
      \frac{\exp(2 \pi h_2 \nu_2)}{|1 - b e^{2\pi \mathrm{i} \nu_2}|^2} \, d \nu_2
    =
    \frac{b^{\pm h_2}}{|1 - b^2|} ,
  \]
  and \eqref{eq:gamma-symmetric} holds true.

  The causality conditions from Proposition~\ref{prop:withpsikl-causal}
  rewrite as follows:
  \begin{align*}
    (1 - a)(1 - b) &> 0, &  (1 - a)(1 + b) &> 0, \\
    (1 + a)(1 - b) &> 0, &  (1 + a)(1 + b) &> 0.
  \end{align*}
  They mean that $1-a$, $1-b$, $1+a$ and $1+b$ should be (nonzero and) of the same sign.
  As sum of these factors $(1-a)+(1-b)+(1+a)+(1+b)=4$ is positive,
  the causality condition is equivalent to 
  \[
    1-a>0, \qquad 1-b>0, \qquad 1+a>0, \quad \mbox{and} \quad 1+b>0,
  \]
  which in turn is equivalent to $|a|<1$ and $|b|<1$.
\end{proof}

\subsubsection{Case $a = -b c$: uncorrelated observations along transect line}
Another special case is where the coefficients in \eqref{eq:difeq1} satisfy relation
$a = - b c$.
\begin{proposition}
  Let $a = -bc$. The stationary field $X$
  and a collection of uncorrelated zero-mean variables with equal variance $\sigma_\epsilon^2$
  which satisfy \eqref{eq:difeq1}
  exist if and only if $|b| \neq 1$ and $|c| \neq 1$.
  The autocovariance function of $X$ satisfies relations
  \[
    \gamma_X(h_1,0) = 0, \quad h_1\neq 0\, ; 
    \qquad
    \gamma_X(0, h_2) = \frac{b^{\pm h_2} \sigma_\epsilon^2}
    {|1-b^2| \times |1-c^2|}
  \]
  where the sign ``$\pm$'' is chosen so that $|b^{\pm h_2}|<1$.

  The field $X$ is causal with respect to $\epsilon$ is and only if $|b|<1$ and $|c|<1$.
\end{proposition}

\begin{proof}
  Since $a=-ac$,
  \begin{align*}
    D &= (1 + bc - b - c) (1 + bc + b + c) (1 - bc - b + c) (1 - bc + b - c)
      = \\ &=
         (1 - b)^2 (1 - c)^2 (1 + b)^2 (1 + c)^2 .
  \end{align*}
  Thus, $D\ge 0$, and the necessary and sufficient condition
  of existence of a stationary solution, namely $D>0$,
  is satisfied if and only if
  $(1-b) (1-c) (1+b) (1+c) \neq 0$,
  which is equivalent to $|b|\neq 1$ and $|c| \neq 1$.

  For the autocovariance function, we use results of Section~\ref{ss:h1eq0}.
  With notation \eqref{eq:def-buni} and \eqref{eq:def-auni},
  \begin{gather*}
    \sqrt{D} = |1 - b^2| \times |1 - c^2|, \qquad
    \alpha = 0, \\
    \begin{aligned}
      \beta &= 
    \frac{2 (-b^2 c + b)}{1 - b^2 c^2 + b^2 - c^2 + \sign(1 - b^2 c^2 + b^2 - c^2) \sqrt{D}}
      = \\ &=
    \frac{2 (1 - c^2) b}{(1+b^2) (1-c^2)  + \sign(1 - c^2) |1 - b^2| \times |1 - c^2|}
      = \\ &= 
    \frac{2 (1 - c^2) b}{(1+b^2) (1-c^2)  + |1 - b^2| \, (1 - c^2)}
      = \\ &=
    \frac{2 b}{(1+b^2)  + |1 - b^2| }
    =
    \begin{cases}
      b & \mbox{if $|b|<1$,} \\
      b^{-1} & \mbox{if $|b|>1$.}
    \end{cases}
    \end{aligned}
  \end{gather*}
  The desired formulas for the autocovariance function follow from \eqref{eq:gX0h2} and \eqref{eq:gXh10}.

  The causality conditions, which are obtained in Proposition~\ref{prop:withpsikl-causal},
  can be rewritten as follows:
  \begin{align*}
    (1 - b)(1 - c) &> 0, &  (1 + b)(1 + c) &> 0, \\
    (1 - b)(1 + c) &> 0, &  (1 + b)(1 - c) &> 0.
  \end{align*}
  These conditions are equivalent to $|b|<1$ and $|c|<1$.
  The proof is similar to one in Proposition~\ref{prop:scsym-acf}.
\end{proof}

The values of the autocovariance function for $a = -0.1$, $b=0.5$,
$c = 0.2$, and $\sigma_\epsilon^2 = 0.72$ are shown in Table~\ref{tab:acfsc2}.
The parameters are chosen in such a way that $\var X_{i,j} = 1$.
\begin{table}
  \caption{The autocovariance function $\gamma_X(h_1,h_2)$
           for $a=-0.1$, $b=0.5$, $c=0.2$, and $\sigma_\epsilon^2=0.72$}
  \begin{center}
    \begin{tabular}{lllllll}
      & \multicolumn{6}{l}{\textbf{Values of }\boldmath$\gamma_{X}(h_1,h_2)$} \\
      \hline
      \boldmath$h_2$ & \boldmath$h_1{=}{-}2$ & \boldmath$h_1{=}{-}1$ &
      \boldmath$h_1{=}0$ & \boldmath$h_1{=}1$ & \boldmath$h_1{=}2$ & \boldmath$h_1{=}3$ \\
      \hline
      \hphantom{\textbf{--}}\textbf{3} & \hphantom{--}0 & 0 & \textbf{0.125} &
        0.1125 & \hphantom{--}0.0225 & --0.00225 \\
      \hphantom{\textbf{--}}\textbf{2} & \hphantom{--}0 & 0 & \textbf{0.25} &
        0.15   & \hphantom{--}0.0075 & --0.003 \\
      \hphantom{\textbf{--}}\textbf{1} & \hphantom{--}0 & 0 & \textbf{0.5} &
        0.15   & --0.015 & \hphantom{--}0.0015 \\
      \hphantom{\textbf{--}}\textbf{0} & \hphantom{--}\textbf{0} & \textbf{0} & \textbf{1} &
          \textbf{0} & \hphantom{--}\textbf{0} & \hphantom{--}\textbf{0} \\
      \textbf{--1} & --0.015 & 0.15 & \textbf{0.5} &
        0 & \hphantom{--}0 & \hphantom{--}0 \\
      \textbf{--2} & \hphantom{--}0.0075 & 0.15 & \textbf{0.25} &
        0 & \hphantom{--}0 & \hphantom{--}0 \\
      \textbf{--3} & \hphantom{--}0.0225 & 0.1125 & \textbf{0.125} &  
        0 & \hphantom{--}0 & \hphantom{--}0 \\
      \hline
    \end{tabular}
  \end{center}
  \label{tab:acfsc2}
\end{table}

\subsection{Nonidentifiability. Symmetry}
Let $X$ be a stationary solution to eq.~\eqref{eq:difeq1}.
Consider the relation between parameters $a$, $b$, $c$ and $\sigma_\epsilon^2$ and
the distribution of the random field $X$ --- more specifically, we study the multiplicity
of this relation.
The situation is complicated by the fact that the distribution of $X$ is not determined uniquely by the
parameters $a$, $b$, $c$ and $\sigma_\epsilon^2$; it also depends on the distribution of the error field
$\epsilon$ besides the scaling coefficient $\sigma_\epsilon$.

\begin{definition}
  Let us have a statistical model $(\Omega, \mathcal{F},\; \mathsf{P}_{\theta,\phi}, \: (\theta,\phi)\in\Psi)$,
  that is a family of probability measures $\mathsf{P}_{\theta,\psi}$ on a common measurable space
  $(\Omega, \mathcal{F})$ is indexed by parameter $(\theta,\phi)\in\Psi$.
  The parameter $\theta$ is called \textit{identifiable} if
  $\mathsf{P}_{\theta_1,\psi_1} = \mathsf{P}_{\theta_2,\psi_2}$ implies $\theta_1 = \theta_2$.
\end{definition}
In the model considered, the parameter of interest $\theta$ comprises
the coefficients and the error variance, $\theta = (a, b, c, \sigma_\epsilon^2) \in \Theta$,
$\sigma_\epsilon>0$,
\[
  \Theta = \{ (a, b, c, v) : 
 (1{-}a{-}b{-}c)(1{-}a{+}b{+}c)(1{+}a{-}b{+}c)(1{+}a{+}b{-}c)>0 \mbox{~and~} v>0 \}.
\]
The nuisance ``non-parametric'' parameter $\phi$ describes the distribution of
the normalized error field $\sigma_\epsilon^{-1} \epsilon$.
In the Gaussian case, $\sigma_\epsilon^{-1} \epsilon$ is a collection of independent random variables all of
standard normal distribution.
In general case, $\sigma_\epsilon^{-1} \epsilon$ is a collection of uncorrelated zero-mean unit-variance random variables.

It is much simpler to consider the autocovariance function $\gamma_X(h_1,h_2)$ than the distribution of $X$,
as $\gamma_X(h_1,h_2)$ is uniquely determined by $\theta$ due to \eqref{eq:acovf-via-integral-expl}.
In the Gaussian case, the distribution of zero-mean stationary field $X$ is uniquely determined
by its autocovariance function; and obviously the distribution uniquely determines the autocovariance function.
In general case, if parameters $\theta_1$ and $\theta_2$ determine the same autocovariance function,
than for any ``distribution of a collection of uncorrelated zero-mean unit-variance variables $\sigma_\epsilon^{-1} \epsilon$''
described by parameter $\phi_1$ there exists a corresponding distribution described by parameter $\phi_2$
such that $P_{\theta_1,\phi_1} = P_{\theta_2,\phi_2}$.
Thus, it is enough to find out which parameters $\theta$ yield the same autocovariance function $\gamma_X(h_1,h_2)$.

In this section we will abuse notation by denoting the random field $X$ as $X_{i,j}$.
Similarly, $X_{-i,-j}$ will be a random field $X^{(2)}$ defined so that $X^{(2)}_{i,j} = X_{-i,-j}$.

By substitution of indices, \eqref{eq:difeq1} can be rewritten as
\begin{gather*}
  X_{-i,-j} = - \frac{b}{c} X_{1-i,-j} - \frac{a}{c} X_{-i,1-j} + \frac{1}{c} X_{1-i,1-j} - \frac{1}{c} \epsilon_{1-i,1-j} , \\
  X_{-i,j} = \frac{1}{a} X_{1-i,j} - \frac{c}{a} X_{-1,j-1} - \frac{b}{a} X_{1-i,j-1} - \frac{1}{a} \epsilon_{1-i,j}, \\
  X_{i,-j} = -\frac{c}{b} X_{i-1,-j} + \frac{1}{b} X_{i,1-j} - \frac{a}{b} X_{i-1,1-j} - \frac{1}{b} \epsilon_{i,1-j}.
\end{gather*}
Thus, random fields $X_{-i,-j}$, $X_{-i,j}$ and $X_{i,-j}$ satisfy equations of the structure of \eqref{eq:difeq1},
but with different coefficients and different error term.  The correspondence of parameters is presented in Table~\ref{tab:cocoef}.

\begin{table}
  \caption{Correspondence of terms of the recurrence equation for fields $X_{i,j}$, $X_{-i,-j}$, $X_{-i,j}$ and $X_{i,-j}$}
  \begin{center}
    \begin{tabular}{lllllll}
      \hline
      \textbf{Line} & \textbf{Field of} & \multicolumn{3}{l}{\textbf{Coefficients}} & \multicolumn{2}{l}{\textbf{Error field}} \\
      \textbf{no., \boldmath$m$} & \textbf{interest} & & & & \textbf{variance} & \textbf{values} \\
      \hline
      1 & $X_{i,j}$   & $ a  $ & $ b  $ & $ c  $ & $            \sigma_\epsilon^2$ & $\epsilon_{i,j}$ \\
      2 & $X_{-i,-j}$ & $-b/c$ & $-a/c$ & $ 1/c$ & $c^{-2}      \sigma_\epsilon^2$ & $-c^{-1} \epsilon_{1-i,1-j}$ \\
      3 & $X_{-i,j}$  & $ 1/a$ & $-c/a$ & $-b/a$ & $a^{-2}      \sigma_\epsilon^2$ & $-a^{-1} \epsilon_{1-i,j}$ \\
      4 & $X_{i,j-1}$ & $-c/b$ & $ 1/b$ & $-a/b$ & $b^{-2}      \sigma_\epsilon^2$ & $-b^{-1} \epsilon_{i,1-j}$ \\
      \hline
    \end{tabular}
  \end{center}
  \label{tab:cocoef}
\end{table}
However the random fields $X_{i,j}$ and $X_{-i,-j}$ have the same autocovariance function,
$\gamma_{X_{i,j}}(h_1,h_2) = \gamma_{X_{-i,-j}}(h_1,h_2)$.
Hence, two sets of parameters $\theta_1 = (a, b, c, \sigma_\epsilon^2)$ and $\theta_2 = (-b/c, -a/c, 1/c, c^{-2}\sigma_\epsilon^2)$
determine the same autocovariance function of the field $X$.

The autocovariance function of random fields $X_{-i,j}$ and $X_{i,-j}$ is flip-symmetric to the autocovariance function
of the field $X_{i,j}$:
\[
  \gamma_{X_{-i,j}}(h_1,h_2) = \gamma_{X_{i,-j}}(h_1,h_2) = \gamma_{X_{i,j}}(-h_1, h_2) =  \gamma_{X_{i,j}}(h_1, -h_2).
\]
In the symmetric case $a b + c = 0$, where the autocovariance function $\gamma_{X_{i,j}}(h_1, h_2)$
is even in each of the arguments, the autocovariance functions of the random fields $X_{i,j}$, $X_{-i,-j}$,
$X_{-i,j}$ and $X_{i,-j}$ are all equal.
Hence, four sets of parameters presented in Table~\ref{tab:cocoefsym} determine the same autocovariance function
of the random fields $X_{i,j}$.
Hence, four sets of parameters $\theta_1 = (a, b, -bc, \allowbreak \sigma_\epsilon^2)$,
$\theta_2 = (a^{-1}!, b^{-1}\!, -a^{-1}b^{-1}\!, \allowbreak a^{-2}b^{-2}\sigma_\epsilon^2)$,
$\theta_3 = (a^{-1}\!, b, -a^{-1}b, \allowbreak a^{-2}\sigma_\epsilon^2)$ and
$\theta_4 = (a, b^{-1}\!, \allowbreak -ab^{-1}\!, \allowbreak b^{-2}\sigma_\epsilon^2)$
determine the same autocovariance function of the random field $X$.
\begin{table}
  \caption{Four sets of parameters $\theta=(a,b,c,\sigma_\epsilon^2)$ that determine the same 
  autocovariance function $\gamma_X(h_1,h_2)$}
  \begin{center}
    \begin{tabular}{llll}
      \hline
       $ a     $ & $ b      $ & $ c=-ab$           &    $              \sigma_\epsilon^2$ \\
       $ a^{-1}$ & $ b^{-1} $ & $ -a^{-1} b^{-1} $ &    $a^{-2} b^{-2} \sigma_\epsilon^2$ \\
       $ a^{-1}$ & $ b      $ & $-a^{-1} b$        &    $a^{-2}        \sigma_\epsilon^2$ \\
       $ a     $ & $ b^{-1} $ & $-a b^{-1}$        &    $b^{-2}        \sigma_\epsilon^2$ \\
      \hline
    \end{tabular}
  \end{center}
  \label{tab:cocoefsym}
\end{table}

Now we prove that there no other spurious cases where different sets of parameters
determine the same autocovariance function.

\begin{lemma}\label{lemma:1of4}
  Let $D>0$, with $D$ defined in Lemma~\ref{lemma:denominator}.
  Then of four sets of parameters listed in Table~\ref{tab:cocoef},
  exactly one set satisfy the causality conditions specified in Proposition~\ref{prop:withpsikl-causal}.
\end{lemma}
The proof involves the search of 7 cases of different signs of $f_1,\ldots,f_4$ defined in \eqref{eq:deff1..f4}, and evaluating inequalities.
For brevity, we do not present the full proof; however, in Table~\ref{tab:7cases4} we list for what signs of $f_1,\ldots,f_4$
which parameterization is causal.
\begin{table}
  \caption{Relation between the signs of $f_1,\ldots,f_4$ and line in Table~\ref{tab:cocoef} where parameters
  satisfy the causality condition}
  \begin{center}
    \begin{tabular}{llll}
      \textbf{Signs of \boldmath$f_1,\ldots, f_4$} &
      \textbf{The} &
      \textbf{Which} &
      \textbf{Outcome} \\
      & \textbf{signs} &
      \textbf{parameter-} & \\
      & \textbf{imply} &
      \textbf{ization} & \\
      & \textbf{that} &
      \textbf{is causal, \boldmath$m$} & \\
      \hline
      $f_1>0$, $f_2>0$, $f_3>0$, $f_4>0$ & & 1 &
      $X_{i,j}$ is causal w.r.t.~$\epsilon_{i,j}$ \\
      \hline
      $f_1<0$, $f_2>0$, $f_3>0$, $f_4<0$ & $c>0$ & 2 &
      $X_{-i,-j}$ is causal
    \\$f_1>0$, $f_2<0$, $f_3<0$, $f_4>0$ & $c<0$ &  & w.r.t.~$\epsilon_{1-i,1-j}$\\
      \hline
      $f_1<0$, $f_2<0$, $f_3>0$, $f_4>0$ & $a>0$ & 3 &
      $X_{-i,j}$ is causal 
      \\
      $f_1>0$, $f_2>0$, $f_3<0$, $f_4<0$ & $a<0$ &  &w.r.t.~$\epsilon_{1-i,j}$ \\
      \hline
      $f_1<0$, $f_2>0$, $f_3<0$, $f_4>0$ & $b>0$ & 4 &
      $X_{i,-j}$ is causal
      \\
      $f_1>0$, $f_2<0$, $f_3>0$, $f_4<0$ & $b<0$ &  & w.r.t.~$\epsilon_{i,1-j}$\\
      \hline
    \end{tabular}
  \end{center}
  \label{tab:7cases4}
\end{table}

\begin{lemma}\label{lemma:covtotheta_causal}
  In the causal case, that is under causality conditions stated in Proposition~\ref{prop:withpsikl-causal},
  the autocovariance function uniquely determines coefficients $a$, $b$ and $c$
  and error variance $\sigma_\epsilon^2$.
\end{lemma}
\begin{proof}
  Using Yule--Walker equations, \eqref{neq:g00more} and \eqref{eq:g1-1},
  we can express the parameters $a$, $b$, $c$ and $\sigma_\epsilon^2$
  in terms of $\gamma_X(0,0)$, $\gamma_X(1,0)$, $\gamma_X(0,1)$ and $\gamma_X(1,1)$.
The explicit formulas are
\begin{gather*}
  a = \frac{\gamma_X(1,0)\gamma_X(0,0)-\gamma_X(0,1)\gamma_X(1,1)}
      {\gamma_X(0,0)^2 - \gamma_X(0,1)^2}, \\
  b = \frac{\gamma_X(0,1)\gamma_X(0,0) - \gamma_X(1,0)\gamma_X(1,1)}
      {\gamma_X(0,0)^2 - \gamma_X(1,0)^2}, \\
  c = \frac{\gamma_X(1,1) - a \gamma_X(0,1) - b \gamma_X(1,0)}{\gamma_X(0,0)}, \\
  \sigma_\epsilon^2 = \gamma_X(0,0) \sqrt{D},
\end{gather*}
where $D$ is defined in Lemma~\ref{lemma:denominator}.
\end{proof}

The following proposition provides a necessary condition for two sets of parameters
to determine the same autocovariance function of the field $X$.
\begin{proposition}\label{prop:necessary-equalacf}
  Let stationary fields $X^{(1)}$ and $X^{(2)}$ satisfy \eqref{eq:difeq1}-like equation
  with parameters $\theta_1 = (a_1,b_1,c_1,\sigma_{\epsilon,1}^2)$ and
  $\theta_2 = (a_2,b_2,c_2,\sigma_{\epsilon,2}^2)$:
  \begin{gather*}
    X^{(1)}_{i,j} = a_1 X^{(1)}_{i-1,j} + b_1 X^{(1)}_{i,j-1} + c_1 X^{(1)}_{i-1,j-1} + \epsilon^{(1)}_{i,j}, \\
    X^{(2)}_{i,j} = a_2 X^{(2)}_{i-1,j} + b_2 X^{(2)}_{i,j-1} + c_2 X^{(2)}_{i-1,j-1} + \epsilon^{(2)}_{i,j}, \\
    \var \epsilon^{(1)}_{i,j} = \sigma_{\epsilon,1}^2, \qquad
    \var \epsilon^{(2)}_{i,j} = \sigma_{\epsilon,2}^2.
  \end{gather*}
  If random fields $X^{(1)}$ and $X^{(2)}$ have the same autocovariance function, then the parameters $\theta_1$ and $\theta_2$
  relate to each other as quadruples of parameters in Table~\ref{tab:cocoef}.
\end{proposition}

  \begin{proof}
    Denote $T_m : \Theta_m \to \Theta$ the operator that transforms the quadruple of parameters in the fist line of Table~\ref{tab:cocoef}
    into one in the $m$th line, $m=1, 2, 3, 4$.  Thus,
    \begin{gather*}
      T_2(a,b,c,\sigma_\epsilon^2) = \textstyle\left(-\frac{b}{c}, -\frac{a}{c}, \frac{1}{c}, c^{-2} \sigma_\epsilon^2\right), \\
      T_3(a,b,c,\sigma_\epsilon^2) = \textstyle\left(\frac{1}{a}, -\frac{c}{a}, -\frac{b}{a}, a^{-2} \sigma_\epsilon^2\right), \\
      T_4(a,b,c,\sigma_\epsilon^2) = \textstyle\left(-\frac{c}{b}, \frac{1}{b}, -\frac{a}{b}, b^{-2} \sigma_\epsilon^2\right),
    \end{gather*}
    and $T_1$ is the identity operator, $T_1(\theta) = \theta$. Here $\Theta_m \subset \Theta$ is the domain where the operator $T_m$
    is well-defined, e.g., $\Theta_2 = \{(a, b,c, \sigma_\epsilon^2) \in \Theta : \allowbreak c\neq 0\}$.
    We have to prove that
    $\theta_1 = T_m \theta_0$ and $\theta_2 = T_n \theta_0$ for some $m, n=1,\ldots,4$ and $\theta_0 \in \Theta$.

    According to Lemma~\ref{lemma:1of4}, one of parameterization $T_1 \theta_1, \ldots, T_4 \theta_1$
    satisfies the conditions for causality; let it be $T_m \theta_1 = \theta_3$. Denote also $X^{(3)}$ and $\epsilon^{(3)}$
    the random field defined by such parameterization, and the respective white noise; what this means is listed in Table~\ref{tab:cocoef}.
    For example, if $m=1$, that $X^{(3)} = X^{(1)}$, and if $m=2$, then $X^{(3)}_{i,j} = X^{(1)}_{-i,-j}$.

    The random fields $X^{(1)}$ and $X^{(3)}$ are either equal or flip or turn symmetric to each other,
    $X^{(3)}_{i,j} = X^{(1)}_{\pm i, \pm j}$.
    Hence, their autocovariance functions are either equal or one-variable symmetric to each other: either
    \[ \gamma_{X^{(3)}} (h_1, h_2) = \gamma_{X^{(1)}} (h_1, h_2) \quad \mbox{for all $h_1$ and $h_2$}, \]
    or
    \[ \gamma_{X^{(3)}} (h_1, h_2) = \gamma_{X^{(1)}} (-h_1, h_2) =\gamma_{X^{(1)}} (h_1, -h_2) \quad \mbox{for all $h_1$ and $h_2$}. \]

    The random field $X^{(3)}$ is causal w.r.t.\ white noise $\epsilon^{(3)}$. Hence, due to Proposition~\ref{prop:gh1h2_ads},
    \begin{equation}\label{eq:X3h1h2-ads}
      \gamma_{X^{(3)}} (h_1, h_2) = \frac{\gamma_{X^{(3)}}(h_1, 0) \gamma_{X^{(3)}}(0, h_2)} {\gamma_{X^{(3)}}(0, 0)} 
      \quad \mbox{if $h_1 h_2 < 0$.}
    \end{equation}

    We do the same with the field $X^{(2)}$. There is $n$ and a field $X^{(4)}$ that satisfy a \eqref{eq:difeq1}-like
    equation with parameters $\theta_4 = T_n \theta_2$; these parameters satisfy the conditions for causality,
    and either
    \begin{gather*}
     \gamma_{X^{(4)}} (h_1, h_2) = \gamma_{X^{(2)}} (h_1, h_2) \quad \mbox{for all $h_1$ and $h_2$, or} \\
     \gamma_{X^{(4)}} (h_1, h_2) = \gamma_{X^{(2)}} (-h_1, h_2) =\gamma_{X^{(1)}} (h_1, -h_2) \quad \mbox{for all $h_1$ and $h_2$}, 
    \end{gather*}
    and also
    \begin{equation}\label{eq:X4h1h2-ads}
      \gamma_{X^{(4)}} (h_1, h_2) = \frac{\gamma_{X^{(4)}}(h_1, 0) \gamma_{X^{(4)}}(0, h_2)} {\gamma_{X^{(4)}}(0, 0)} 
      \quad \mbox{if $h_1 h_2 < 0$.}
    \end{equation}

    Similar relations hold for autocovariance functions of $X^{(3)}$ and $X^{(4)}$: either
    \begin{gather*}
     \gamma_{X^{(4)}} (h_1, h_2) = \gamma_{X^{(3)}} (h_1, h_2) \quad \mbox{for all $h_1$ and $h_2$, or} \\
     \gamma_{X^{(4)}} (h_1, h_2) = \gamma_{X^{(3)}} (-h_1, h_2) =\gamma_{X^{(1)}} (h_1, -h_2) \quad \mbox{for all $h_1$ and $h_2$}, 
    \end{gather*}
    This implies $\gamma_{X^{(4)}} (h_1, h_2) = \gamma_{X^{(3)}} (h_1, h_2)$ if  $h_1 =0$ or $h_2=0$, and because of
    \eqref{eq:X3h1h2-ads} and \eqref{eq:X4h1h2-ads},
    \[
      \gamma_{X^{(4)}} (h_1, h_2) = \gamma_{X^{(3)}} (h_1, h_2) \quad \mbox{if $h_1 h_2 < 0$.}
    \]
    All above implies that random fields $X^{(3)}$ and $X^{(4)}$ have the same autocovariance function.
    These fields are causal w.r.t.\ respective white noises. According to Lemma~\ref{lemma:covtotheta_causal},
    their parameters are equal, $\theta_3 = \theta_4$.

    The operators $T_1,\ldots,T_4$ are self-inverse; $T_k^2$ is the identity operator.
    Thus, $T_m \theta_1 = \theta_3 = \theta_4 = T_n \theta_2$ implies $\theta_1 = T_n \theta_3$
    and $\theta_2 = T_m \theta_4$. Thus, parameters $\theta_1$ and $\theta_2$ relate
    to each other as the $m$th and $n$th rows of Table~\ref{tab:cocoef}.
  \end{proof}

  The following corollary combines necessary and sufficient conditions for parameters to
  determine the same autocovariance function; it also takes into account where the parameters make sense. 
  \begin{corollary}
    The parameter set $\Theta$ splits into two-element, four-element and one-element classes 
    of parameters that define the same autocovariance function $\gamma_X(h_1,h_2)$ as follows:
    \begin{enumerate}
      \item
        In generic case $-ab \neq c \neq 0$, the parameter $(a,b,c,\sigma_\epsilon^2)$ belongs to a two-element class.
        It determines the same autocovariance function as the parameter $(-b/c, -a/c, 1/c, c^{-2} \sigma_\epsilon^2)$.
      \item
        In symmetric case $-ab = c \neq 0$, the parameter $(a,b,-bc,\sigma_\epsilon^2)$ belongs to a four-element class.
        All four sets of parameters listed in Table~\ref{tab:cocoefsym} determine the same autocovariance function.
      \item
        In symmetric case $a\neq 0$, $b=c=0$, the parameter $(a,0,0,\sigma_\epsilon^2)$ belongs to a two-element class.
        It determines the same autocovariance function as the parameter $(a^{-1}, 0, 0, a^{-2} \sigma_\epsilon^2)$.

        Similarly, if $b\neq 0$, then the parameter $(0,b,0,\sigma_\epsilon^2)$ belongs to two-element class
        and determines the same autocovariance function as the parameter $(0,b^{-1},0,b^{-2}\sigma_\epsilon^2)$.
      \item
        For $-ab \neq c = 0$, the parameter $(a, b, 0, \sigma_\epsilon^2)$ makes a class of its own.
        Likewise, for $a=b=c=0$, the parameter $(0, 0,0, \sigma_\epsilon^2)$ makes a class of its own.
        These are exceptional cases where the autocovariance function uniquely determines the parameters.
    \end{enumerate}
  \end{corollary}

  Next proposition shows how the asymmetric case is split into two major subcases.
  \begin{proposition}\label{prop:3cases}
    Let $X$ be a stationary field and $\epsilon$ be a collection of zero-mean equal-variance random variables
    that satisfy \eqref{eq:difeq1}.
    \begin{enumerate}
        \item
    If $ab + c = 0$, then
    \[
      \gamma_X(h_1, h_2) = \frac{\gamma_X(h_1, 0) \gamma_X(0,h_2)} {\gamma_X(0,0)} 
        \quad \mbox{for all $h_1$ and $h_2$}.
    \]
      \item
    If $ab + c \neq 0$ and $1+c^2 > a^2+b^2$, then
        \begin{gather*}
            \gamma_X(1, -1) = \frac{\gamma_X(1, 0) \gamma_X(0,1)} {\gamma_X(0,0)} \neq \gamma_X(1, 1), \\
      \gamma_X(h_1, h_2) = \frac{\gamma_X(h_1, 0) \gamma_X(0,h_2)} {\gamma_X(0,0)} 
        \quad \mbox{for all $h_1$ and $h_2$ such that $h_1 h_2 \le 0$}.
        \end{gather*}
      \item
    If $ab + c \neq 0$ and $1+c^2 < a^2+b^2$, then
        \begin{gather*}
            \gamma_X(1, -1) \neq \frac{\gamma_X(1, 0) \gamma_X(0,1)} {\gamma_X(0,0)} = \gamma_X(1, 1), \\
      \gamma_X(h_1, h_2) = \frac{\gamma_X(h_1, 0) \gamma_X(0,h_2)} {\gamma_X(0,0)} 
        \quad \mbox{for all $h_1$ and $h_2$ such that $h_1 h_2 \ge 0$}.
        \end{gather*}
    \end{enumerate}
  \end{proposition}
  
  \begin{proof}
    The symmetric case $ab + c = 0$ has been studied in Section~\ref{sss:symmetric_case}.
    The desired equality follows from Proposition~\ref{prop:scsym-acf}.

    Let us borrow some notation from the proof of Proposition~\ref{prop:necessary-equalacf},
    with $X^{(1)} = X$.
    Let $\theta_3 = (a_3, b_3, c_3, \sigma_{\epsilon,3}^2) = T_m (a,b,c,\sigma_\epsilon^2)$
    be a set of parameters obtained from Lemma~\ref{lemma:1of4} that satisfy the causality conditions
    and let $X^{(3)}$ be a stationary random field defined by this parameterization;
    $X^{(3)}$ is causal w.r.t.\ the respective white noise.
    The sign of
    \[
      1 + c^2 - a^2 - b^2 = \frac{f_1 f_4 + f_2 f_3}{2}
    \]
    determines the relation between the autocovariance functions of fields $X$ and $X^{(3)}$.
    If $1+c^2 > a^2 + b^2$, then either $X^{(3)} =  X$ or $X^{(3)}_{i,j} = X_{-i,-j}$;
    in either case the fields $X$ and $X^{(3)}$ have the same autocovariance function.
    Otherwise, if $1+c^2 < a^2 + b^2$, then either $X^{(3)}_{i,j} = X_{-i,j}$
    or $X^{(3)}_{i,j} = X_{i,-j}$; in both cases the autocovariance functions are
    one-variable symmetric, $\gamma_{X^{(3)}}(h_1,h_2) =  \gamma_{X}(h_1,-h_2)$.

    As $X^{(3)}$ is causal, due to Proposition~\ref{prop:gh1h2_ads}
    \[
      \gamma_{X^{(3)}}(h_1,h_2) = \frac{\gamma_{X^{(3)}}(h_1,0)\gamma_{X^{(3)}}(0,h_2)}
      {\gamma_{X^{(3)}}(0,0)}
      \quad \mbox{if $h_1 h_2 \le 0$} .
    \]
    Hence, and from the relation between $\gamma_{X}$ and $\gamma_{X^{(3)}}$
    all the desired equalities follow.

    Now prove the inequalities. Assume that the would-be inequality is actually an equality, 
    that is $\gamma_X(1,1) = \gamma_X(1, 0) \gamma_X(0,1) / \gamma_X(0,0)$ in case $1+c^2>a^2 + b^2$
    or $\gamma_X(1,-1) = \gamma_X(1, 0) \gamma_X(0,1) / \gamma_X(0,0)$ in case $1+c^2<a^2 + b^2$.
    Then the autocovariance function of the field $X^{(3)}$ satisfies
    \[
      \gamma_{X^{(3)}}(1, 1) = \frac{\gamma_{X^{(3)}}(1, 0) \gamma_{X^{(3)}}(0,1)} {\gamma_{X^{(3)}}(0,0)}.
    \]
    The field $X^{(3)}$ is causal, and formulas in Lemma~\ref{lemma:covtotheta_causal}
    yield
    \[
      a_3 = \frac{\gamma_{X^{(3)}}(1,0)}{\gamma_{X^{(3)}}(0,0)}, \qquad
      b_3 = \frac{\gamma_{X^{(3)}}(0,1)}{\gamma_{X^{(3)}}(0,0)}, \qquad
      c_3 = -\frac{\gamma_{X^{(3)}}(1,0)\gamma_{X^{(3)}}(0,1)}{\gamma_{X^{(3)}}(0,0)^2}
          = - a_3 b_3 .
    \]
    Examining 4 cases, we can verify a similar relation for the original parameterization, $c = -a b$.
  \end{proof}

  \section{Pure nondeterminism}
  \label{sec:pnd}
  \subsection{Sufficient condition}
    We borrow the definition of a pure nondeterministic random field from \cite{Tjostheim:1978}.
    For the field on a planar lattice, the definition rewrites as follows,
    in term of subspaces of the Hilbert space of finite-variance random variables $L^2(\Omega, \mathcal{F}, \mathsf{P})$.
    \begin{definition}
        Let $X$ be a field of random variables of finite variances.
        $X$ is called \textit{purely nondeterministic} if and only if
        \begin{align}
          \closedspan_{i,j} X_{i,j} = \closedspan_{i,j} \Bigl(
            & \closedspan\{X_{r,s} : r\le i\} \cap
              (\closedspan\{X_{r,s} : r<i\})^\bot
            \cap \mbox{} \nonumber \\ &\quad \cap
              \closedspan\{X_{r,s} : s \le j\} \cap
              (\closedspan\{X_{r,s} : s < j\})^\bot \Bigr) .
          \label{eq:defpndT}      
        \end{align}
     \end{definition}
     Here $\closedspan_{i,j} X_{i,j}$ is the smallest (closed) subspace of $L^2(\Omega, \mathcal{F}, \mathsf{P})$
     that contains all observations of the field $X$, while $\closedspan\{X_{r,s} : r\le i\}$
     is a similar minimal subspace which contains all observations $X_{r,s}$ of the field $X$
     with first index $r\le i$. The outer $\closedspan$ in the right-hand size of \eqref{eq:defpndT}
     is the smallest subspace that contains all subspaces used as an argument.

     Pure nondeterminism is a sufficient condition for a stationary field to representable in form 
     \eqref{eq:causal_eq}, where $\epsilon$ is some collection of uncorrelated variables $\epsilon$,
     and coefficients $\psi_{k,l}$ satisfy $\sum_{k=0}^\infty \sum_{l=0}^\infty \psi_{k,l}^2 < \infty$.
     This follows from \cite[Theorem~2.1]{Tjostheim:1978}, while the necessity is easy to verify.

     \begin{proposition}
       Let $X$ be a stationary field and $\epsilon$ be a collection of
       uncorrelated zero-mean unit-variance random variables that satisfy \eqref{eq:difeq1}.
       If $1 + c^2 > a^2 + b^2$, then $X$ is purely nondeterministic.
     \end{proposition}

     \begin{proof}
       We use notation $D$ from Lemma~\ref{lemma:denominator} and notation
       $f_1,\ldots,f_4$ from \eqref{eq:deff1..f4}.

       The existence of a stationary solution to \eqref{eq:deff1..f4} imply that $D>0$.
       As shown in the proof of Proposition~\ref{prop:3cases}, $D>0$ and $1+c^2 > a^2 + b^2$
       holds true if three cases, $f_1>0$, $f_2>0$, $f_3>0$, $f_4>0$
       or $f_1<0$, $f_2>0$, $f_3>0$, $f_4<0$ or $f_1>0$, $f_2<0$, $f_3<0$, $f_3<0$.

       If $f_1>0$, $f_2>0$, $f_3>0$, $f_4>0$,
       then the field $X$ is causal with respect to white noise $\epsilon$.
       Equations \eqref{eq:difeq1} and \eqref{eq:causal_eq} imply that
       \[
         \closedspan\{X_{r,s} : r\le i\} = \closedspan{\epsilon_{r,s} : r\le i}, \quad\;
         \closedspan\{X_{r,s} : s\le j\} = \closedspan{\epsilon_{r,s} : s\le j},
       \]
       whence
       \begin{multline*}
             \closedspan\{X_{r,s} : r\le i\} \cap
              (\closedspan\{X_{r,s} : r<i\})^\bot
            \cap \mbox{} \\ \cap 
            \closedspan\{X_{r,s} : s \le j\} \cap
              (\closedspan\{X_{r,s} : s < j\})^\bot
              = \colspan \epsilon_{i,j} .
       \end{multline*}
       Thus,
       \begin{multline*}
          \closedspan_{i,j} \Bigl(
             \closedspan\{X_{r,s} : r\le i\} \cap
              (\closedspan\{X_{r,s} : r<i\})^\bot
            \cap \mbox{} \\ \cap
              \closedspan\{X_{r,s} : s \le j\} \cap
              (\closedspan\{X_{r,s} : s < j\})^\bot \Bigr) .
          =
          \closedspan_{i,j} \epsilon_{i,j} = \closedspan_{i,j} X_{i,j},
       \end{multline*}
       and \eqref{eq:defpndT} holds true.

       Now consider two other cases, where $-f_1$, $f_2$, $f_3$ and $-f_4$ are nonzero and of the same case.
       In these cases $c\neq 0$, and
       \[
       \tilde \epsilon_{i,j} = X_{i,j} + \frac{b}{c} X_{i-1,j} + \frac{a}{c} X_{i,j-1} - \frac{1}{c} X_{i-1,j-1}
       \]
       is also a collection of uncorrelated zero-mean equal-variance random variables
       (to verify this, one can compute the spectral density of $\tilde\epsilon$).
       Causality condition in Proposition~\ref{prop:withpsikl-causal} can be  easily verified;
       the stationary field $X$ is causal w.r.t.\ white noise $\tilde\epsilon$.
       The field $X$ is causal in these cases likewise.
     \end{proof}

  \subsection{Counterexample to Tj{\o}stheim}
  It would be tempting to use the sufficient condition for pure nondeterminism
  from \cite[Theorem 3.1]{Tjostheim:1978}; however, that condition is not correct.
  We construct a counterexample.

  Let $0<|\theta| < 1$.
  Let $X$ be a stationary field that satisfies the equation
  \[
    X_{i-1,j} = \theta X_{i,j-1} + \epsilon_{i,j},
  \]
  where $\epsilon$ is a collection of independent random variables
  with standard normal distribution, $\epsilon_{i,j} \sim \mathcal{N}(0,1)$.

  The spectral density of the field $X$ is
  \[
    f_X(\nu_1,\nu_2) = \frac{1}
    {|e^{2\pi \mathrm{i} \nu_1} - \theta e^{2 \pi \mathrm{i} \nu_2}
    |^2} \, .
  \]
  The denominator $|e^{2\pi \mathrm{i} \nu_1} - \theta e^{2 \pi \mathrm{i} \nu_2}|^2$
  attains only positive values and is a continuous function.
  Thus, the field $X$ satisfies the sufficient condition stated
  in \cite[Theorem~3.1]{Tjostheim:1978}.

  In the filed $X$ the diagonals $\{X_{i,j}, \; i+j=n\}$
  are independent for different $n$.
  The random variables on each diagonal are jointly distributed as
  values of a centered Gaussian AR(1) process with the coefficient $\theta$.

  Let
  \[
    \tilde \epsilon_{i,j} = X_{i,j-1} - \theta X_{i-1,j}.
  \]
  Then $\tilde\epsilon$ is a collection of uncorrelated zero-mean unit-variance variables.
  (Since $\tilde\epsilon$ is a Gaussian field, it is a collection independent variables with distribution
  $\mathcal{N}(0,1)$.)

  The field $X$ can be represented as 
  \[
    X_{i,j} = \sum_{k=0}^\infty \theta^k \epsilon_{i+k+1,j-k} = \sum_{k=0}^\infty \theta^k \tilde\epsilon_{i-k, j+k+1} .
  \]
  Hence,
  \[
    \tilde\epsilon_{i,j} = 
    - \theta \epsilon_{i,j} + (1 - \theta^2) \sum_{k=0}^\infty \epsilon_{i+k+1,j-k-1} 
  \]
  These representations imply that
  \[
    \closedspan\{X_{r,s} : r \le i\} = \closedspan\{\tilde\epsilon_{r,s} : r \le i\}, \qquad
    \closedspan\{X_{r,s} : s \le j\} = \closedspan\{\epsilon_{r,s} : s \le j\} .
  \]
  Hence
       \begin{multline*}
             \closedspan\{X_{r,s} : r\le i\} \cap
              (\closedspan\{X_{r,s} : r<i\})^\bot
            \cap \mbox{} \\ \cap 
            \closedspan\{X_{r,s} : s \le j\} \cap
              (\closedspan\{X_{r,s} : s < j\})^\bot
              = \closedspan_s \tilde\epsilon_{i,s} \cap \closedspan_r \epsilon_{r,j} .
       \end{multline*}
  Now prove that $\closedspan_s \tilde\epsilon_{i,s} \cap \closedspan_r \epsilon_{r,j}$
  is a trivial subspace. Let $\zeta \in \closedspan_s \tilde\epsilon_{i,s} \cap \closedspan_r \epsilon_{r,j}$,
  \[
    \zeta = \sum_{s=-\infty}^\infty k_s \tilde\epsilon_{i,s} = \sum_{r=-\infty}^\infty c_r \epsilon_{r,j} .
  \]
  Here the coefficients satisfy $\sum_{s=-\infty}^\infty  k_s^2 = \sum_{r=-\infty}^\infty c_r^2 < \infty$;
  the series converge in mean squares.
  The covariance between $\tilde\epsilon_{i,s}$ and $\epsilon_{r,j}$ is
  \[
    \ME\tilde\epsilon_{i,s} \epsilon_{r,j} = \begin{cases}
      -\theta & \mbox{if $i=r$ and $s=j$,} \\
      \theta^{r-i-1} (1 - \theta^2) & \mbox{if $r-i = s-j > 0$,} \\
      0 & \mbox{otherwise.}
    \end{cases}
  \]
  Anyway, $\ME\tilde\epsilon_{i,s} \epsilon_{r,j} = 0$ if $r-i \neq s-j$, 
  and $|\ME\tilde\epsilon_{i,s} \epsilon_{r,j}| < 1$ if $r-i = s-j$.

  Compute $\ME \tilde\epsilon_{i,j-n} \zeta$ and $\ME \zeta \epsilon_{i-n,j}$ using two different expressions:
  \begin{gather*}
    \sum_{s=-\infty}^\infty k_s \ME \tilde\epsilon_{i,j-n} \tilde\epsilon_{i,s} =\ME \tilde\epsilon_{i,j-n} \zeta =
    \sum_{r=-\infty}^\infty c_r \ME \tilde\epsilon_{i,j-n} \epsilon_{r,j}, \\
    \sum_{s=-\infty}^\infty k_s \ME \tilde\epsilon_{i,s} \epsilon_{i-n,j}=\ME \zeta \epsilon_{i-n,j} =
    \sum_{r=-\infty}^\infty c_r \ME \epsilon_{r,j} \epsilon_{i-n,j}.
  \end{gather*}
  In the series, the only nonzero term might be where $s=j-n$ and $r=i-n$. Thus,
  \begin{equation}\label{eq:2eqs-2}
    k_{j-n} = c_{i-n} \ME \tilde\epsilon_{i,j-n} \epsilon_{i-n,j}, \quad
    k_{j-n} \ME \tilde\epsilon_{i,j-n} \epsilon_{i-n,j} = c_{i-n} .
  \end{equation}
  As $|\ME \tilde\epsilon_{i,j-n} \epsilon_{i-n,j}| < 1$, \eqref{eq:2eqs-2} imply $k_{j-n} = c_{i-n} = 0$.
  As this holds true for all integer $n$, $k_s = c_r = 0$ for all $s$ and $r$,
  and $\zeta = 0$ almost surely.
  Thus,
       \begin{multline*}
             \closedspan\{X_{r,s} : r\le i\} \cap
              (\closedspan\{X_{r,s} : r<i\})^\bot
            \cap \mbox{} \\ \cap 
            \closedspan\{X_{r,s} : s \le j\} \cap
              (\closedspan\{X_{r,s} : s < j\})^\bot
              = \{0\},
       \end{multline*}
  and the right-hand side of \eqref{eq:defpndT} is the trivial subspace.
  Thus, the random field $X$ is not purely nondeterministic.

\section{Conclusion}
\label{sec:conclusion}
We considered AR(1) model on a plane.
We found conditions (in terms of the regression coefficients)
under which the autoregressive equation has a stationary solution $X$.
As for the autocovariance function of the stationary solution $X$,
we presented a simple formula for it at some points, and proved
Yule--Walker equations.
These allow to compute the autocovariance function recursively at all points.

We found conditions under which the stationary solution to
the autoregressive equation satisfies the causality condition with respect
to the underlying white noise.
These conditions also appear to be sufficient conditions for
stability of the deterministic problem of solving a recursive
equations in a quadrant, with preset values on the border of the quadrant.

We described sets of parameters (the coefficients and the variance of
the underlying white noise) where different parameters determine the same autocovariance function of the stationary solution.

We found sufficient conditions for the stationary solution $X$ to be a pure nondeterministic random field. This condition is related to the causality condition  with respect to some (non-fixed) white noise, which is called innovations; in particular, the innovations need not coincide with the white noise in the autoregressive equation.

The causality condition and pure-nondeterministic property seem to be too restrictive because the coordinate-wise order between coordinates of
$\epsilon$ and $X$ in the representation \eqref{eq:causal_eq} is a partial order. More general representation with lexical order,
which is a total order,
is suggested in \cite[Section~6]{Whittle:1954}:
\[
  X_{i,j} = \sum_{k=1}^\infty \sum_{l=-\infty}^\infty \psi_{k,l} X_{i-k,j-l} + \sum_{l=0}^\infty \psi_{0,l} X_{i,j-l} .
\]
Further discussion can be found in \cite{Kallianpur:1981,Tjostheim:1983}.

With exception for the identifiability topic, we did not consider the estimation at all. 

Most results of this paper are necessary and sufficient conditions for some properties of the stationary field $X$. However, for pure nondeterminism of $X$ and for stability of the deterministic equation, we obtained only sufficient conditions. Obtaining necessary conditions is an interesting open problem.

\section{Appendix: Auxiliary results}
\label{sec:appendix}
\subsection{Existence of a random field with given spectral density}
Next lemma is used on the proof of Proposition~\ref{prop:existance-eX-newps}.
It was a modification of the similar result stated for stochastic processes \cite{ShumSt}.
\begin{lemma}\label{lem:existence-field-dessdens}
  Let $f(\nu_1, \nu_2)$ be an even integrable function $\left[-\frac12,\frac12\right]^2
  \to [0,\infty]$,
  that is
  \begin{gather*}
    f(\nu_1,\nu_2) = f(-\nu_1, -\nu_2) \ge 0  \quad \mbox{for all $\nu1,\nu_2\mathbin{\in} [-\frac12, \frac12]$}, \\
    \iint_{[-1/2,\,1/2]^2} f(\nu_1, \nu_2) \, d\nu_1 \, d\nu_2 < \infty .
  \end{gather*}
  Then there exists a Gaussian stationary field $\{X_{i,j}, \; i,j\mathbin{\in}\mathbb{Z}\}$
  on some (specially constructed) probability space that has spectral density $f(\nu_1, \nu_2)$.
\end{lemma}

\begin{proof}
  Let us construct a probability space with
  two Brownian fields on a plane $\{W_k(t,s),\; t,s\mathbin{\in}[0,1]\}$
  (a zero-mean Gaussian field with covariance function $\cov(W_k(t_1,s_1),\, W_k(t_2,s_2)) = \min(t_1,t_2) \min(s_1,s_2)$
  \quad $k=1,\, 2$).

  Extend the function $X$ by periodicity, $f(\nu_1, \nu_2) = f(\nu_1-1, \nu_2)$ if $\frac12 < \nu_1 \le 1$,
  and then $f(\nu_1, \nu_2-1) = f(\nu_1-1, \nu_2)$ if $\frac12 < \nu_2 \le 1$.

  Denote
  \begin{align*}
    X_{i,j} &= \iint_{[0,1]^2} \cos(2\pi(i \nu_1 + j \nu_2))
    \sqrt{f(\nu_1,\nu_2)} \, d^2 W_1(\nu_1, \nu_2) + \mbox{} \\ &+
               \iint_{[0,1]^2} \sin(2\pi(i \nu_1 + j \nu_2))
    \sqrt{f(\nu_1,\nu_2)} \, d^2 W_2(\nu_1, \nu_2).
  \end{align*}
  The constructed field is zero-mean Gaussian as the integral of nonstochastic kernel with respect Gaussian field.
  The autocovariance function is as expected,
  \begin{align*}
    \cov(X_{i,j}, X_{i+h_1,j+h_2})
      &=
      \iint_{[0,1]^2} \cos(2 \pi (h_1 \nu_1 + h_2 \nu_2)) f(\nu_1, \nu_2) \, d\nu_1 \, d\nu_2
      = \\ &=
      \iint_{[-1/2,1/2]^2} \exp(2 \pi \mathrm{i} (h_1 \nu_1 + h_2 \nu_2)) f(\nu_1, \nu_2) \, d\nu_1 \, d\nu_2 .
  \end{align*}
\end{proof}

\subsection{Summation}
The next two propositions are used implicitly when we use double-series notation.
\begin{proposition}\label{prop:absolute-double-series}
  Let $\{\xi_{i,j},\;  i, j\in\mathbb{Z}\}$ be a collection of random variables
  bounded in mean squares, $\sup_{i,j} \ME \xi_{i,j}^2 < \infty$.
  Let $\{\xi_{i,j},\;  i, j\in\mathbb{Z}\}$ be a collection of real numbers such that
  ${\sum\sum}_{i,j=-\infty}^\infty |a_{i,j}| < \infty$.
  Then the double series ${\sum\sum}_{i,j=-\infty}^\infty a_{i,j} \xi_{i,j}$
  converges in mean squares and almost surely. The limit it the same for both types of convergence
  (up to equal-almost-surely equivalence);
  it also does not depend on the order the terms of the double series are added.
\end{proposition}
\begin{remark}\mbox{}
  \begin{enumerate}
    \item
      Under conditions of Proposition~\ref{prop:absolute-double-series},
      the double series ${\sum\sum}_{i,j=-\infty}^\infty a_{i,j} \xi_{i,j}$
      converges in mean and in probability, as well --- to the same limit.
    \item
      The iterated series $\sum_{i=-\infty}^\infty \left(\sum_{j=-\infty}^\infty a_{i,j} \xi_{i,j}\right)$
      converges to the same limit.
    \item
      $\Prob\left({\sum\sum}_{i,j=-\infty}^\infty |a_{i,j} \xi_{i,j}|<\infty\right) = 1$,
      and whenever ${\sum\sum}_{i,j=-\infty}^\infty |a_{i,j} \xi_{i,j}|<\infty$,
      the double series ${\sum\sum}_{i,j=-\infty}^\infty a_{i,j} \xi_{i,j}$ converges, and its limit
      does not depend on the order its terms are added.
  \end{enumerate}
\end{remark}

\begin{proposition}\label{prop:unconditional-double-series}
  Let $\{\epsilon_{i,j},\;  i, j\in\mathbb{Z}\}$ be a collection of random uncorrelated variables
  with zero mean and same variance $\var \epsilon_{i,j} = \sigma_\epsilon^2 < \infty$.
  Let $\{\xi_{i,j},\;  i, j\in\mathbb{Z}\}$ be a collection of real numbers such that
  ${\sum\sum}_{i,j=-\infty}^\infty a_{i,j}^2 < \infty$.
  Then the double series ${\sum\sum}_{i,j=-\infty}^\infty a_{i,j} \xi_{i,j}$
  converges in mean squares.
  The limit does not depend on the order the terms of the double series are added:
  it is the same up to equal-almost-surely equivalence.
\end{proposition}
\pagebreak
\begin{remark}\mbox{}
  \begin{enumerate}
    \item
      Under conditions of Proposition~\ref{prop:unconditional-double-series},
      the double series ${\sum\sum}_{i,j=-\infty}^\infty a_{i,j} \xi_{i,j}$
      converges in mean and in probability, as well.
    \item
      The iterated series $\sum_{i=-\infty}^\infty \left(\sum_{j=-\infty}^\infty a_{i,j} \xi_{i,j}\right)$
      converges in mean squares to the same limit.
  \end{enumerate}
\end{remark}

Next lemma is used in Proposition~\ref{prop:withpsikl-causal}.
\begin{lemma}\label{lemma:suf-four-summ}
  Let $f\in C(\mathbb{R}^2)$ be a biperiodic function of two variables with continuous mixed derivative:
  \begin{gather*}
    f(x,y) = f(x+1,y) = f(x, y+1) \quad \mbox{for all $x,y\in\mathbb{R}$},
    \\
    f\in C(\mathbb{R}^2), \quad
    \frac{\partial f}{\partial x}\in C(\mathbb{R}^2), \quad
    \frac{\partial f}{\partial y}\in C(\mathbb{R}^2), \quad
    \frac{\partial^2 f}{\partial x \, \partial y}
    \in C(\mathbb{R}^2) .
  \end{gather*}
  Then Fourier coefficients of the function $f$ are summable:
  \[
    \sum_{k=-\infty}^\infty \sum_{l=-\infty}^\infty
    \left| \int_0^1 \int_0^1 \exp(2 \pi \mathrm{i} \, (k x + l y) ) \, f(x,y) \, dx \, dy \right|
    < \infty .
  \]
\end{lemma}
\begin{proof}
  Denote the Fourier coefficients $c_{k,l}$:
  \[
    c_{k,l} =
    \int_0^1 \int_0^1 \exp(2 \pi \mathrm{i} \, (k x + l y) ) \, f(x,y) \, dx \, dy .
  \]
  The coefficients $c_{k,l}$ are also equal
  \begin{gather*}
    c_{k,l} = \frac{\mathrm{i}}{2\pi k} \int_0^t \int_0^1 \exp(2 \pi \mathrm{i} \, (k x + l y) ) \, f'_1(x,y) \, dx \, dy
    \qquad\mbox{if $k\neq 0$}; \\
    c_{k,l} = \frac{\mathrm{i}}{2\pi l} \int_0^t \int_0^1 \exp(2 \pi \mathrm{i} \, (k x + l y) ) \, f'_2(x,y) \, dx \, dy
    \qquad\mbox{if $l\neq 0$}; \\
    c_{k,l} = \frac{-1}{4 \pi^2 k l} \int_0^t \int_0^1 \exp(2 \pi \mathrm{i} \, (k x + l y) ) \, f''_{12}(x,y) \, dx \, dy
    \qquad\mbox{if $k\neq 0$ and $l\neq 0$},
  \end{gather*}
  where $f'_1(x,y)$, $f'_2(x,y)$, and $f''_{12}(x,y)$
  are partial and mixed derivatives;
  here the periodicity of the function $f$ is used.
  The coefficients $c_{k,l}$ allow bounding:
  \begin{gather*}
    c_{k,0} \le \frac{1}{2\pi |k|} \max_{x,y} \left| \frac{\partial f(x,y)}{\partial x} \right| \qquad\mbox{if $k\neq 0$}; \\
    c_{0,l} \le \frac{1}{2\pi |l|} \max_{x,y} \left| \frac{\partial f(x,y)}{\partial y} \right| \qquad\mbox{if $l\neq 0$}; \\
    c_{k,l} \le \frac{1}{4\pi^2 |k l|} \max_{x,y} \left| \frac{\partial^2 f(x,y)}{\partial y \partial x} \right| \qquad\mbox{if $k\neq 0$ and $l\neq 0$},
  \end{gather*}
  which implies the summability.
\end{proof}

\subsection{Integration}
\begin{lemma}\label{lemma:integration}
  If $A>0$, $B$ and $C$ are real numbers such that $A^2 >  B^2 + C^2$,
  then
  \begin{equation}\label{eq:li-eq1}
    \int_{-1/2}^{1/2} \frac{d t}
    {A + B \cos(2 \pi t) + C \sin(2 \pi t)}
    = \frac{1}{\sqrt{A^2 - B^2 - C^2}} .
  \end{equation}

  If $A$ and $B$ are real numbers, $A > |B|$, and $n$ is integer, then
  \begin{equation}\label{eq:li-eq2}
    \int_{-1/2}^{1/2} \frac{\exp(2 \pi \mathrm{i} n t)}
    {A + B \cos (2 \pi t)} \, dt
    = \frac{\alpha^{|n|}}{\sqrt{A^2 - B^2}},
  \end{equation}
  where
  \[
    \alpha = \frac{-B}{A + \sqrt{A^2 - B^2}}
    = \begin{cases}
      0 & \mbox{if $B=0$}, \\
      (-A + \sqrt{A^2 - B^2}) / B & \mbox{if $B\neq 0$} .
    \end{cases}
  \]
\end{lemma}
In \eqref{eq:li-eq2}, if $\alpha=0$ and $n=0$, then $\alpha^{|n|} = 1$ by convention.
\begin{proof}
  That is easy to check that the antiderivative in \eqref{eq:li-eq1} is
  \begin{multline*}
    \int\frac{dt}{A + B \cos(2\pi t) + C \sin(2\pi t)}
    = \\ =
    \frac{1}{\pi \sqrt{A^2 - B^2 - C^2}} \,
    \arctan\!\left( \frac{(A - B) \tan(\pi t) + C}{\sqrt{A^2 - B^2 - C^2}} \right) + \mathrm{const},
  \end{multline*}
  whence \eqref{eq:li-eq1} follows. (Notice that $A-B > 0$.)

  In \eqref{eq:li-eq2}, the function
  \[
    \frac{1}{A + B \cos(2 \pi t)} = \frac{1 + \alpha^2}{A \, (1 - 2 \alpha \cos (2\pi t) + \alpha^2)}
    = \frac{1 + \alpha^2}{A} \, \frac{1}{|1 - \alpha \exp(2 \pi \mathrm{i} t)|^2}
  \]
  is the spectral density of AR(1) stationary autoregressive process $X_k = \alpha X_{k-1} + \epsilon_k$
  with white-noise variance $\var \epsilon_k = (1 + \alpha^2) / A$. (Conditions imply that $|\alpha| < 1$.)
  The integral is the autocovariance function of the process. It equals
  \begin{equation*}
    \int_{-1/2}^{1/2} \frac{\exp(2 \pi \mathrm{i} n t)}
    {A + \cos (2 \pi t)} \, dt
    = \cov(X_{k+n}, X_k) = \frac{1 + \alpha^2}{A} \, \frac{\alpha^{|n|}}{1 - \alpha^2}
    = \frac{\alpha^{|n|}}{\sqrt{A^2 - B^2}} \, ;
  \end{equation*}
  here equality
  \[
    \frac{1+\alpha^2}{1-\alpha^2} = \frac{A}{\sqrt{A^2 - B^2}}
  \]
  is used.
\end{proof}

\begin{lemma}\label{lem:int-oneie}
  If $A$ and $B$ are complex numbers, $|A| \neq |B|$,
  and $n$ is an integer number, then
  \begin{equation}\label{eq:intetyui}
    \int_{-1/2}^{1/2}
    \frac{e^{2\pi{\mathrm i} n \nu}}{A - B e^{2\pi {\mathrm{i}} \nu}} \, d\nu
    =
    \begin{cases}
      A^{n-1} B^{-n} & \mbox{if $n\le 0$ and $|A| > |B|$,} \\
      0              & \mbox{if $n\le 0$ and $|A| < |B|$,} \\
      0              & \mbox{if $n\ge 1$ and $|A| > |B|$,} \\
      -A^{n-1} B^{-n} & \mbox{if $n\ge 1$ and $|A| < |B|$.} \\
    \end{cases}
  \end{equation}
\end{lemma}
\begin{proof}
  For $n=1$, the integral can be rewritten as the contour integral,
  \[
    \int_{-1/2}^{1/2}
    \frac{e^{2\pi{\mathrm i} \nu}}{A - B e^{2\pi {\mathrm{i}} \nu}} \, d\nu
    =
    \frac{1}{2\pi \mathrm{i}} \oint_{z=\exp(2 \pi \mathrm{i} \nu)}
    \frac{1}{A - B z} \, d z
  \]
  and the residue formula is applicable.
  Other cases can be reduced to the case $n=1$ recursively.
\end{proof}

\begin{lemma}\label{lem:int-oneie2}
  Let $a$, $b$ and $c$ be real numbers such that
  $1-a-b-c > 0$, {} $1-a+b+c > 0$, {} $1+a-b+c>0$ and $1+a+b-c > 0$.
  Let $n_1$ and $n_2$ be nonnegative integers.
  Then
  \[
    \int_{-1/2}^{1/2}
    \frac{(a + c e^{2 \pi \mathrm{i} \nu})^{n_1} e^{-2\pi \mathrm{i} n_2 \nu}}
    {(1 - b e^{2\pi\mathrm{i} \nu})^{n_1+1}}
    \, d \nu =
    \sum_{k=0}^{\min(n_1,n_2)}
    \binom{n_1}{k} \binom{n_2}{k} a^{n_1-k} b^{n_2-k} (ab+c)^k.
  \]
\end{lemma}
\begin{proof}
  Conditions of the lemma imply that
  \begin{gather*}
    1-b = \frac{(1-a-b-c)+(1+a-b+c)}{2} > 0, \\
    1+b = \frac{(1-a+b+c)+(1+a+b-c)}{2} > 0, \\
    |b| < 1.
  \end{gather*}

  Let $k\ge 0$ and $n$ be integer numbers.
  By binomial formula,
  \begin{equation}\label{eq:e1}
    \frac{1}{(1 - b z)^{k+1}} = \sum_{m=0}^\infty \binom{-k-1}{m} (- b z)^m
    = \sum_{m=0}^{\infty} \binom{k+m}{m} (b z)^m
  \end{equation}
  for all complex $z$ such that $|b z| < 1$, in particular, for  all $z$ such that $|z|\le 1$.
  Here
  $\binom{-k-1}{m} = \frac{(-k-1)\cdots(-k-m)}{1\cdots m} = (-1)^m \binom{k+m}{m}$
  is a coefficient in a binomial series.
    The rational function $z^n (1 - b z)^{-k-1}$
  have a singularity a point $1/b$ and potential singularity (if $n<0$) at point $0$.
  The point $1/b$ lies outside the unit circle, and point 0 lies inside the unit circle.
  The expansion \eqref{eq:e1} implies that
  \begin{align*}
    \Res_{z=0} \left( \frac{z^n}{(1 - b z)^{k+1}} \right)
    &=
    \Res_{z=0} \left( \sum_{m=0}^\infty \binom{k+m}{m} b^m z^{m+n} \right)
    = \\ &=
    \begin{cases}
      0 & \mbox{if $n\ge 0$}, \\
      \binom{k-n-1}{-n-1} b^{-n-1}
      & \mbox{if $n \le -1$ .}
    \end{cases}
  \end{align*}
  By the residue formula,
  \begin{align*}
    \oint \frac{z^n}{(1-b z)^{k+1}} \, dz
    &=
    2 \pi \mathrm{i} \Res_{z=0} \left( \frac{z^n}{(1 - b z)^{k+1}} \right)
    = \\ &=
    \begin{cases}
      0 & \mbox{if $n\ge 0$}, \\
      2 \pi \mathrm{i} \binom{k-n-1}{-n-1} b^{-n-1}
      & \mbox{if $n \le -1$,}
    \end{cases}
  \end{align*}
  where the integral is taken along the unit-circle contour.

  Recall that $n_2 \ge 0$ is integer.
  Then
  \begin{multline}\label{eq:int-trif-cases}
    \int_{-1/2}^{1/2} \frac{\exp(2 \pi \mathrm{i} (k-n_2) \nu_2)}
      {(1 - b e^{2\pi\mathrm{i} \nu_2})^{1+k}} \, d\nu_2
    =
    \frac{1}{2\pi\mathrm{i}} \oint \frac{z^{k - n_2 - 1}} {(1 - b z)^{k+1}} \, dz
    = \\ =
    \begin{cases}
      0 & \mbox{if $k>n_2$,} \\
      \binom{n_2}{n_2-k} b^{n_2-k}
      & \mbox{if $k\le n_2$}
    \end{cases}
    =
    \begin{cases}
      0 & \mbox{if $k>n_2$,} \\
      \binom{n_2}{k} b^{n_2-k}
      & \mbox{if $k\le n_2$.}
    \end{cases}
  \end{multline}

  The constant $n_1\ge 0$ is also integer, and by binomial formula
  \begin{align}
    (a + c e^{2 \pi\mathrm{i} \nu})^{n_1}
    &=
    ((a b + c) e^{2 \pi \mathrm{i} \nu} + a (1 - b e^{2 \pi \mathrm{i} \nu}))^{n_1}
    = \nonumber \\ &=
    \sum_{k=0}^{n_1}
    \binom{n_1}{k} (a b + c)^k e^{2 \pi \mathrm{i} k \nu}
                   a^{n_1-k} (1 - b e^{2 \pi \mathrm{i} \nu})^{n_1 - k}.
  \label{eq:binom-use-2}
  \end{align}

  Finally, with use of \eqref{eq:binom-use-2} and \eqref{eq:int-trif-cases},
  \begin{align*}
    & \int_{-1/2}^{1/2}
      \frac{e^{-2 \pi \mathrm{i} n_2 \nu} (a + c e^{2\pi \mathrm{i} \nu})^{n_1}}
           {(1 - b e^{2\pi \mathrm{i} \nu})^{n_1+1}}
           \, d\nu
      = \\ &=
      \int\limits_{-1/2}^{1/2}\!\!
      \frac{e^{-2 \pi \mathrm{i} n_2 \nu}}
           {(1 {-} b e^{2\pi \mathrm{i} \nu})^{n_1+1}}
           \sum_{k=0}^{n_1} \!
           \binom{n_1}{k}
      (a b {+} c)^k e^{2\pi \mathrm{i} k\nu} a^{n_1-k} (1 {-} b e^{2\pi \mathrm{i} \nu})^{n_1-k}
            d\nu
      = \\ &=
      \sum_{k=0}^{n_1}
      \binom{n_1}{k}
      a^{n_1 - k}
      (a b + c)^k
      \int_{-1/2}^{1/2}
      \frac{e^{2 \pi \mathrm{i} (k-n_2) \nu}}
           {(1 - b e^{2\pi \mathrm{i} \nu})^{1+k}}
           \, d\nu
      = \\ &=
      \sum_{k=0}^{\min(n_1,\,n_2)}
      \binom{n_1}{k}
      a^{n_1 - k}
      (a b + c)^k
      \binom{n_2}{k} b^{n_2-k} .
      \qedhere
  \end{align*}
\end{proof}

\iffalse
  Stoffer  Shumway
  Quenouille
  Biometrika
  Baran Pap Zuijlen
  Cressie
  Dunford
  Springer
  Kabaila
  Unwin
  Hepple
\fi


\begin{thebibliography}{8}
  \bibitem{Baran:2004}
    S.~Baran, G.~Pap, M.~C.~A.~van Zuijlen.
    Asymptotic inference far a nearly unstable
    sequences of stationary spatial AR models.
    \textit{A Journal of Theoretical and Applied Statistics},
    2004, Vol.~38, No.~6, 465--482.\\
    doi:10.1080/02331880412331319297

  \bibitem{Cressie:1993}
    N.~A.~C.~Cressie.
    Statistics for Spatial Data.
    Wiley, 1993.\\
    doi:10.1002/9781119115151

  \bibitem{Jain:1981}
    A.~K.~Jain.
    Advances in mathematical models for image processing.
    \textit{Proceedings of the IEEE}, 1981, Vol.~69, No.~5, 502--528. \\
    doi:10.1109/PROC.1981.12021



  \bibitem{Kallianpur:1981}
    G.~Kallianpur.
    Some remarks of the purely nondeterministic property
    of second-order random fields.
    In \textit{Stochastic Differential Systems}, Springer, 1981, pp.~98--109.
    doi:10.1007/BFb0006413


  \bibitem{Martin:1979}
    R.~J.~Martin.
    A subclass of lattice processes applied to a problem in planar sampling.
    \textit{Biometrika}, 1979, Vol.~66, No.~2, 209--217.\\
    doi:10.2307/2335651

  \bibitem{Martin:1990}
    R.~J.~Martin.
    The use of time-series models and methods in the analysis of agricultural field trials.
    \textit{Communications in Statistics\,---\,\allowbreak
    Theory and Methods}, 1990, Vol.~19, No.~1, pp.~55--81.
    doi:10.1080/03610929008830187

  \bibitem{Martin:1996}
    R.~J.~Martin.
    Some results on unilateral ARMA lattice processes.
    \textit{Journal of Statistical Planning and Inference}, 1996, Vol.~50, No.~3, 395--411.\\
    doi:10.1016/0378-3758(95)00066-6

  \bibitem{Pickard:1980}
    D.~K.~Pickard.
    Unilateral Markov fields.
    \textit{Advances in Applied Probability}, 1980, Vol.~12, No.~3, 655--671.

  \bibitem{ShumSt}
    R.~H.~Shumway, D.~S.~Stoffer.
    Time Series Analysis and Its Applications: With R Examples.
    Springer, New York, 2011.

  \bibitem{Tjostheim:1978}
    D.~Tj{\o}stheim.
    Statistical spatial series modelling.
    \textit{Advances in Applied Probability}, 1978, Vol.~10, No.~1, 130--154.
    doi:10.2307/1426722

  \bibitem{Tjostheim:1983}
    D.~Tj{\o}stheim.
    Statistical spatial series modelling II.
     Some further results on unilateral lattice processes.
    \textit{Advances in Applied Probability}, 1983, Vol.~15, No.~3, 562--584.
    doi:10.2307/1426619

  \bibitem{ToryPickard:1992}
    E.~M.~Tory and D.~K.~Pickard.
    Unilateral Gaussian fields.
    \textit{Advances in Applied Probability}, 1992, Vol.~24, No.~1, pp.~95--112.
    doi:10.2307/1427731


  \bibitem{UnwinHepple:1974}
    D.~J.~Unwin, L.~W.~Hepple.
    The statistical analysis of spatial series.
    \textit{Journal of the Royal Statistical Society: Series D},
    1974, Vol.~23, No.~3/4, 211--227.
    doi:10.2307/2987581



  \bibitem{Whittle:1954}
    P.~Whittle.
    On stationary processes in the plane.
    \textit{Biometrika}, 1954, Vol.~41, No.~3--4, 434--449.
    doi:10.1093/biomet/41.3-4.434


\end{thebibliography}
\end{document}